\documentclass[12pt,letterpaper]{article}
\usepackage{amsmath,amsfonts,amsthm,amssymb}

\swapnumbers 

\newtheorem{lemma}{Lemma}[section]
\newtheorem{corollary}[lemma]{Corollary}
\newtheorem{theorem}[lemma]{Theorem}
\newtheorem{example}[lemma]{Example}

\newtheorem{assumption}[lemma]{Standing Assumption}

\theoremstyle{definition} 
\newtheorem{definition}[lemma]{Definition}

\newtheorem{remarks}[lemma]{Remarks}

\newcommand\reals{{\mathbb R}}

\newcommand{\dg}{\sp{\text{\rm o}}}

\begin{document}

\title{Two projections in a synaptic algebra}

\author{David J. Foulis,{\footnote{Emeritus Professor, Department of
Mathematics and Statistics, University of Massachusetts, Amherst,
MA; Postal Address: 1 Sutton Court, Amherst, MA 01002, USA;
foulis@math.umass.edu.}}\hspace{.05 in} Anna Jen\v cov\'a and Sylvia
Pulmannov\'{a} {\footnote{ Mathematical Institute,
Slovak Academy of Sciences, \v Stef\'anikova 49, SK-814 73 Bratislava,
Slovakia; pulmann@mat.savba.sk, jenca@mat.savba.sk. \newline The second
and third authors were supported by Research and Development Support Agency
under the contract No. APVV-0178-11 and grant VEGA 2/0059/12.}}}

\date{}

\maketitle

\begin{abstract}
\noindent We investigate P. Halmos' \emph{two projections theorem},
(or \emph{two subspaces theorem}) in the context of a synaptic
algebra (a generalization of the self-adjoint part of a von Neumann
algebra).
\end{abstract}

{\small \noindent \emph{Mathematics subject Classification.} 17C65, 81P10, 47B15

\noindent \emph{Keywords:} Synaptic algebra, projection, Peirce decomposition, generic position, CS-decomposition}

\section{Introduction}
In what follows, $A$ is a synaptic algebra with enveloping algebra $R
\supseteq A$, \cite{FSynap, FPSynap, TDSA, SymSA, ComSA, PuNote} and $P$
is the orthomodular lattice \cite{Beran, Kalm} of projections in $A$. For
instance, if ${\mathcal B({\mathcal H})}$ is the algebra of all bounded
linear operators on the Hilbert space ${\mathcal H}$ and ${\mathcal A}$
is the set of all self-adjoint operators in ${\mathcal B}({\mathcal H)})$,
then ${\mathcal A}$ is a synaptic algebra with enveloping algebra
${\mathcal B}({\mathcal H})$. See the literature cited above for numerous
additional examples of synaptic algebras.

In this article, we show that Halmos' work \cite{Halmos} on two projections
on (or two subspaces of) a Hilbert space can be generalized to the context
of the synaptic algebra $A$.

A leisurely, lucid, and extended exposition of Halmos' theory of two
projections can be found in the paper \cite{Guide} of A. B\"{o}ttcher
and I.M. Spitkovsky, where the basic theorem \cite[Theorem 1.1]{Guide}
is expressed in terms of linear subspaces of a Hilbert space, projections
onto these linear subspaces, and operator matrices. Working only with our
synaptic algebra $A$, we have to forgo both Hilbert space and the
operator matrix calculus---still we shall formulate generalizations of
\cite[Theorem 2]{Halmos}, often called \emph{Halmos's CS-decomposition
theorem}. (See Theorems \ref{th:genCSdecomp}, \ref{th:HalmosCStheorem},
and Section \ref{sc:OpMat} below). Also, in Section \ref{sc:Apps},
we give a brief indication of how our version of the CS-decomposition
can be used to develop analogues for synaptic algebras of some of the
consequences of Halmos' theorem for operator algebras.

\section{The orthomodular lattice of projections} \label{sc:OMLP}

In this section we outline some of the notions and facts pertaining to
the synaptic algebra $A$ and to the orthomodular lattice $P\subseteq A$
that we shall need in this article. In what follows, we shall use these
notions and facts routinely, often without attribution. More details and
proofs can be found in \cite{FSynap, FPSynap, TDSA, SymSA, ComSA, PuNote}.
We use the symbol $:=$ to mean `equals by definition,' as usual we
abbreviate `if and only if' by `iff,' and the ordered field of real numbers
is denoted by $\reals$.

If $a,b\in A$, then the product $ab$, calculated in the enveloping
algebra $R$, may or may not belong to $A$. However, if $ab=ba$, i.e.,
if $a$ commutes with $b$ (in symbols $aCb$), then $ab\in A$. Also, if
$ab=0$, then $aCb$ and $ba=0$. We define $C(a):=\{b\in A:aCb\}$. If
$aCc$ whenever $c\in A$ and $cCb$, we say that $a$ \emph{double
commutes} with $b$, in symbols $a\in CC(b)$.

There is a unit element $1\in A$ such that $a=a1=1a$ for all $a\in A$.
To avoid trivialities, we assume that $A$ is \emph{nondegenerate}, i.e.,
that $1\not=0$.

Let $a,b,c\in A$. Than, although $ab$ need not belong to $A$, it
turns out that $ab+ba\in A$. Likewise, although $abc$ need not
belong to $A$, it can be shown that $abc+cba\in A$.

The synaptic algebra $A$ is a partially ordered real linear space under
the partial order relation $\leq$ and we have $0<1$ (i.e., $0\leq 1$
and $0\not=1$); moreover, $1$ is a (strong) order unit in $A$. Elements
of the ``unit interval" $E:=\{e\in A:0\leq e\leq 1\}$ are called
\emph{effects}, and $E$ is a so-called \emph{convex effect algebra}
\cite{GPBB}.

If $0\leq a\in A$, then there is a uniquely determined element $r\in A$
such that $0\leq r$ and $r\sp{2}=a$; moreover, $r\in CC(a)$ \cite[Theorem 2.2]
{FSynap}. Naturally, we refer to $r$ as the \emph{square root} of $a$, in
symbols, $a\sp{1/2}:=r$. If $b\in A$, then $0\leq b\sp{2}$, and the \emph
{absolute value} of $b$ is defined and denoted by $|b|:=(b\sp{2})\sp{1/2}$.
Clearly, $|b|\in CC(b)$ and $|-b|=|b|$. Also, if $aCb$, then $|a|C|b|$ and
$|ab|=|a||b|$.

Elements of the set $P:=\{p\in A:p=p\sp{2}\}$ are called \emph{projections}
and it is understood that $P$ is partially ordered by the restriction of
$\leq$. The set $P$ is a subset of the convex set $E$ of effects in $A$; in
fact, $P$ is the extreme boundary of $E$ (\cite[Theorem 2.6]{FSynap}).
Evidently, $0,1\in P$ and $0\leq p\leq 1$ for all $p\in P$. It turns out
that $P$ is a lattice, i.e., for all $p,q\in P$, the \emph{meet} (greatest
lower bound) $p\wedge q$ and the \emph{join} (least upper bound) $p\vee q$
of $p$ and $q$ exist in $P$; moreover, $p\leq q$ iff $pq=qp=p$. Two
projections $p$ and $q$ are called \emph{complements} iff $p\wedge q=0$
and $p\vee q=1$.

Let $p,q\in P$. The \emph{orthocomplement} of $p$, defined by $p\sp{\perp}
:=1-p$, is again an element of $P$, and we have the following: $0\sp
{\perp}=1$, $1\sp{\perp}=0$, $p\sp{\perp\perp}=p$, $p\leq q\Rightarrow q
\sp{\perp}\leq p\sp{\perp}$, $p\wedge p\sp{\perp}=pp\sp{\perp}=0$, and $p
\vee p\sp{\perp}=p+p\sp{\perp}=1$. Furthermore, $p\leq q$ iff $q-p
\in P$, in which case $q-p=q\wedge p\sp{\perp}=qp\sp{\perp}=p\sp{\perp}q$.
Also, we have the \emph{De\,Morgan laws}: $(p\wedge q)\sp{\perp}=p
\sp{\perp}\vee q\sp{\perp}$ and $(p\vee q)\sp{\perp}=p\sp{\perp}\wedge q
\sp{\perp}$.

The projections $p$ and $q$ are said to be \emph{orthogonal}, in symbols
$p\perp q$, iff $p\leq q\sp{\perp}$. The \emph{orthosum} $p\oplus q$ is
defined iff $p\perp q$, in which case $p\oplus q:=p+q$. It turns out that
$p\perp q\Leftrightarrow pCq$ with $pq=qp=0$; furthermore, $p\perp q
\Rightarrow pCq$ with $p\oplus q=p+q=p\vee q\in P$. The lattice $P$,
equipped with the orthocomplementation $p\mapsto p\sp{\perp}=1-p$, is a
so-called \emph{orthomodular lattice} (OML) \cite{Beran, Kalm}.

\begin{definition} \label{df:generic}
Following Halmos \cite[p. 381]{Halmos}, we shall say that two projections
$p,q\in P$ are in \emph{generic position} iff
\[
p\wedge q=p\wedge q\sp{\perp}=p\sp{\perp}\wedge q=p\sp{\perp}\wedge q
 \sp{\perp}=0,
\]
or equivalently (De\! Morgan) iff
\[
p\vee q=p\vee q\sp{\perp}=p\sp{\perp}\vee q=p\sp{\perp}\vee q
 \sp{\perp}=1.
\]
\end{definition}

If $p\in P$ and $e\in E$, then $e\leq p$ iff $e=ep$ iff $e=pe$ \cite[Theorem
2.4]{FSynap}. Applying this result to the projection $1-p$ and the effect
$1-e$, we deduce that $p\leq e$ iff $p=ep$ iff $p=pe$. In particular, if
$p,q\in P$, then $p\leq q$ iff $p=pq$ iff $p=qp$.

As is well-known, for projections $p,q\in P$, the question of whether or
not $pCq$ can be settled (in various ways) purely in terms of lattice
operations in $P$. For instance,
\[
pCq\Leftrightarrow p=(p\wedge q)\vee(p\wedge q\sp{\perp}).
\]
Using this fact, we obtain the following theorem.

\begin{theorem} \label{th:rpandcommutativity}
Let $p,q\in P$ and define $p\sb{r}:=p\wedge(p\sp{\perp}\vee q)
\wedge(p\sp{\perp}\vee q\sp{\perp})\in P$. Then{\rm:}
\begin{enumerate}
\item $p=((p\wedge q)\vee(p\wedge q\sp{\perp}))\oplus p\sb{r}=
 (p\wedge q)\vee(p\wedge q\sp{\perp})\vee p\sb{r}=(p\wedge q)
 \oplus(p\wedge q\sp{\perp})\oplus p\sb{r}$.
\item $0\leq p\sb{r}\leq p$ and $p-p\sb{r}=(p\wedge q)\vee
 (p\wedge q\sp{\perp})$.
\item $pCq$ iff $p\sb{r}=0$.
\end{enumerate}
\end{theorem}

\begin{proof}
We have both $p\wedge q\leq p$ and $p\wedge q\sp{\perp}\leq p$,
whence $(p\wedge q)\vee(p\wedge q\sp{\perp})\leq p$ and
(De\! Morgan)
\[
p-((p\wedge q)\vee(p\wedge q\sp{\perp}))=p\wedge((p\wedge q)
 \vee(p\wedge q\sp{\perp}))\sp{\perp}=p\wedge(p\sp{\perp}
 \vee q\sp{\perp})\wedge(p\sp{\perp}\vee q)
\]
\[
=
p\wedge(p\sp{\perp}\vee q)\wedge(p\sp{\perp}\vee q\sp{\perp})=
p\sb{r},
\]
whereupon
\[
p=((p\wedge q)\vee(p\wedge q\sp{\perp}))\oplus p\sb{r}=
 (p\wedge q)\vee(p\wedge q\sp{\perp})\vee p\sb{r}.
\]
Also, $p\wedge q\leq q$ and $p\wedge q\sp{\perp}\leq q\sp{\perp}$,
so $(p\wedge q)\perp(p\wedge q\sp{\perp})$ and $(p\wedge q)\vee
(p\wedge q\sp{\perp})=(p\wedge q)\oplus(p\wedge q\sp{\perp})$,
whence $p=(p\wedge q)\oplus(p\wedge q\sp{\perp})\oplus p\sb{r}$,
completing the proof of (i).

That $0\leq p\sb{r}\leq p$ is clear, $p-p\sb{r}=(p\wedge q)
\vee (p\wedge q\sp{\perp})$ follows from (i), and we have (ii).
Part (iii) is a consequence of (ii) and the fact that $pCq
\Leftrightarrow p=(p\wedge q)\vee(p\wedge q\sp{\perp})$.
\end{proof}

In view of Theorem \ref{th:rpandcommutativity} (iii), we can regard
the projection $$p\sb{r}:=p\wedge(p\sp{\perp}\vee q)\wedge(p\sp{\perp}
\vee q\sp{\perp})$$ as a sort of measure of the extent to which $p$
commutes with $q$. Indeed, $pCq$ iff $p\sb{r}=0$; also $0\leq p\sb{r}
\leq p$ and if $p\not=0$, then in some sense, the ``larger" $p\sb{r}$
is, the ``greater the lack of commutativity of $p$ and $q$," culminating
in the case in which $p\sb{r}=p$. By Theorem \ref{th:rpandcommutativity}
(ii), $p\sb{r}=p$ iff $(p\wedge q)\vee(p\wedge q\sp{\perp})=0$ iff
$p\wedge q=p\wedge q\sp{\perp}=0$.

An alternative measure of the extent to which $p$ commutes with $q$,
the \emph{Marsden commutator} $[p,q]$, was introduced by E.L. Marsden,
Jr. in \cite{Marsden}.

\begin{definition} \label{df:commutator}
If $p,q\in P$, then
$$[p,q]:=(p\vee q)\wedge(p\vee q\sp{\perp})\wedge(p\sp{\perp}\vee q)
 \wedge(p\sp{\perp}\vee q\sp{\perp}).$$
\end{definition}
\noindent For the Marsden commutator, we have $0\leq [p,q]\leq 1$, and
as is proved in \cite{Marsden}, $pCq\Leftrightarrow[p,q]=0$. The
relationship between $[p,q]$ and the projection $p\sb{r}$ in Theorem
\ref{th:rpandcommutativity} is explicated in Section \ref{sc:2projcom}
below. We note that $[p,q]$ is as large as possible, i.e., $[p,q]=1$,
iff $p$ and $q$ are in generic position, a situation which is studied
in Sections \ref{sc:Generic} and \ref{sc:dropdown} below.

We recall some additional basic facts regarding commutativity in $P$. Let
$p,q,r\in P$. If $pCq$, then $p\wedge q=pq=qp$ and $p\vee q=p+q-pq$. Also,
$pCq$ iff $pCq\sp{\perp}$, and if either $p\leq q$ or $p\perp q$, then
$pCq$. Furthermore, if $pCq$ and $pCr$, then $pC(q\vee r)$ and $pC
(q\wedge r)$.  Calculations in the OML $P$ are facilitated by the
following theorem \cite[Theorem 5, p. 25]{Kalm} which we use routinely
in what follows:
\begin{theorem} \label{th:distributive}
For $p,q,r\in P$, if any two of the relations $pCq$, $pCr$, or $qCr$
hold, then $p\wedge(q\vee r)=(p\wedge q)\vee(p\wedge r)$ and $p\vee
(q\wedge r)=(p\vee q)\wedge(p\vee r)$.
\end{theorem}

If $a,b\in A$, then it turns out that $aba\in A$, whence we define
the \emph{quadratic mapping} $J\sb{a}\colon A\to A$ by $J\sb{a}b:=aba$
for all $b\in A$. The quadratic mapping $J\sb{a}$ is linear and order
preserving on $A$.

A \emph{synaptic automorphism} on $A$ is a mapping $J\colon A\to A$
such that (1) $J$ is a bijection, (2) $J$ is an order automorphism
on $A$, (3) $J$ is a linear automorphism on $A$, and for all $a,b\in A$,
(4) $ab\in A$ iff $JaJb\in A$ and (5) $ab\in A\Rightarrow J(ab)=JaJb$.

An element $u\in A$ is called a \emph{symmetry} \cite{SymSA} iff
$u\sp{2}=1$, and a \emph{partial symmetry} is an element $t\in A$
such that $t\sp{2}\in P$. By the uniqueness theorem for square roots,
a projection is the same thing as a partial symmetry $p$ such that
$0\leq p$. Each partial symmetry $t\in A$ has a \emph{canonical extension}
to a symmetry $u:=t+(t\sp{2})\sp{\perp}$. If $u$ is a symmetry, then the
quadratic mapping $J\sb{u}$, called a \emph{symmetry transformation}, is
a synaptic automorphism of $A$ and $J\sb{u}\sp{-1}=J\sb{u}$. If $u$ is a
symmetry, then so is $-u$, and $J\sb{u}=J\sb{-u}$. If $u$ and $v$ are
symmetries, then so are $J\sb{u}v=uvu$ and $J\sb{v}u=vuv$. By the uniqueness
theorem for square roots, if $u$ is a symmetry, then $0\leq u\Leftrightarrow
u=1$.

Two projections $p,q\in P$ are \emph{exchanged by a symmetry} $u\in A$
iff $J\sb{u}p=upu=q$ (whence, automatically, $J\sb{u}q=uqu=p$, $J
\sb{u}p\sp{\perp}=up\sp{\perp}u=q\sp{\perp}$, and $J\sb{u}q\sp{\perp}
=uq\sp{\perp}u=p\sp{\perp}$). If $p$ and $q$ are exchanged by a symmetry
$u$, then they are also exchanged by the symmetry $-u$. The two projections
$p$ and $q$ are \emph{exchanged by a partial symmetry} $t\in A$ iff $tpt=q$
and $tqt=p$. If $p$ and $q$ are exchanged by a partial symmetry $t$ and
if $u:=t+(t\sp{2})\sp{\perp}$ is the canonical extension of $t$ to a symmetry,
then $p$ and $q$ are exchanged by the symmetry $u$.

Let $a\in A$. Then there is a uniquely determined projection $a\dg\in P$,
called the \emph{carrier} of $a$, such that, for all $b\in A$, $ab=0
\Leftrightarrow a\dg b=0$. It turns out that $a=aa\dg=a\dg a$, $a\dg\in
CC(a)$, and $a\dg$ is the smallest projection $p\in P$ such that $a=ap$
(or, equivalently, $a=pa$). If $n$ is a positive integer, then $(a\sp{n})
\dg=|a|\dg=a\dg$. Furthermore, if $b\in A$ and $0\leq a\leq b$, then $a
\dg\leq b\dg$.

By \cite[Definition 4.8, Theorem 4.9 (v), and Theorem 5.6]{FSynap}, we
have the following result which we shall need in the proof of Lemma
\ref{lm:diag,offdiag,zero} below and then later in Section \ref
{sc:sinecosine}.

\begin{lemma} \label{lm:carofpqp}
Let $p,q\in P$ and let $0\leq a\in A$. Then{\rm:}
\begin{enumerate}
\item $(J\sb{p}a)\dg=(pap)\dg=(pa\dg p)\dg=p\wedge(p\sp{\perp}\vee a\dg)$.
\item $(pqp)\dg=p\wedge(p\sp{\perp}\vee q)$.
\end{enumerate}
\end{lemma}

We shall also have use for the next two results which follow from
\cite[Lemma 4.1]{SymSA} and \cite[Theorem 5.5]{ComSA}.

\begin{lemma} \label{lm:carrierofsum}
If $0\leq a\sb{1}, a\sb{2},..., a\sb{n}\in A$, then
\[
(\sum\sb{i=1}\sp{n}a\sb{i})\dg=\bigvee\sb{i=1}\sp{n}(a\sb{i})\dg.
\]
\end{lemma}

\begin{lemma} \label{lm:carrierofprod}
If $a,b,ab\in A$, then $(ab)\dg=a\dg b\dg=b\dg a\dg=a\dg\wedge b\dg$.
\end{lemma}

If $a\in A$, there is a partial symmetry $t\in A$, called the \emph
{signum} of $a$, such that $t\sp{2}=a\dg$, $t\in CC(a)$, $a=|a|t=t|a|$,
and $|a|=ta=at$. If $u:=t+(a\dg)\sp{\perp}$ is the canonical extension
of $t$ to a symmetry, then $u\in CC(a)$, $a=|a|u=u|a|$, and $|a|=ua=au$.
The formula $a=|a|u=u|a|$ is referred to as the \emph{polar decomposition}
of $a$.

In Section \ref{sc:sinecosine}, we shall also need the following theorem
\cite[Theorem 6.5]{FPSynap}.
\begin{theorem} \label{th:6.5FPSynap}
Let $p,q\in P$ and let $p-q\sp{\perp}=|p-q\sp{\perp}|u=u|p-q\sp{\perp}|$
be the polar decomposition of $p-q\sp{\perp}$, so that $u$ is a symmetry
that double commutes with $p-q\sp{\perp}$. Then{\rm: (i)} upqpu=qpq.
{\rm(ii)} $u(p\wedge(p\sp{\perp}\vee q))u=q\wedge(p\vee q\sp{\perp})$.
\end{theorem}

Consider the synaptic algebra ${\mathcal A}$ of self-adjoint operators
on a Hilbert space ${\mathcal H}$. If $B\in{\mathcal A}$, then the
carrier $B\dg$ of $B$ is the projection onto the closure of the range
of $B$. Thus, $B\dg=I$, the identity operator on ${\mathcal H}$, iff
$B$ is injective and the range of $B$ is dense in ${\mathcal H}$.
Also, a symmetry $U\in{\mathcal A}$ is the same thing as a
self-adjoint unitary operator on ${\mathcal H}$. Halmos' work in
\cite{Halmos} involves unitary operators mapping one linear subspace
${\mathcal M}$ of ${\mathcal H}$ onto another linear subspace
${\mathcal N}$ of ${\mathcal H}$. Assuming that ${\mathcal M}$ and
${\mathcal N}$ are closed, let $P\sb{\mathcal M}$ and $P\sb
{\mathcal N}$ be the (orthogonal) projections onto ${\mathcal M}$
and ${\mathcal N}$, respectively, and suppose that $U$ is a symmetry
in ${\mathcal A}$ that exchanges $P\sb{\mathcal M}$ and $P\sb
{\mathcal N}$.  Then the restriction $U|{\mathcal M}$ of $U$ to
${\mathcal M}$ is a unitary isomorphism from ${\mathcal M}$ onto
${\mathcal N}$. Thus, in our synaptic algebra $A$, the situation in
which a symmetry $u$ exchanges a projection $p$ with a projection
$q$ may be regarded as an analogue of a situation in which two closed
subspaces of a Hilbert space are unitarily equivalent via the restriction
of a self-adjoint unitary operator.

If $n$ is a positive integer and there are $n$, but not $n+1$ pairwise
orthogonal nonzero projections in $P$, we say that the synaptic algebra
$A$ has \emph{rank} $n$. On the other hand, if there is an infinite
sequence of pairwise orthogonal nonzero projections in $P$, then we say
that $A$ has \emph{infinite rank}. If $A$ is finite dimensional, then
it has finite rank, but there are infinite dimensional synaptic algebras
of finite rank.

\begin{remarks} \label{ex:spinfactor}
It is not difficult to see that \emph{$A$ is of rank 2 iff every pair of
distinct nonorthogonal projections in $P\setminus\{0,1\}$ is in generic
position}. In \cite[\S 19]{Top65}, D. Topping introduced the important notion
of a \emph{spin factor} and, using the results in \cite{FPSpin}, it can be
shown that a Topping spin factor of dimension greater than 1 is the same
thing as synaptic algebra of rank 2. We note that there are infinite-dimensional
synaptic algebras of rank 2.
\end{remarks}

If $r\in P$, then with the partial order and operations inherited
from $A$,
\[
rAr:=J\sb{r}(A)=\{J\sb{r}a:a\in A\}=\{rar:a\in A\}=\{b\in A:b=br=rb\}
\]
is a synaptic algebra in its own right with $rRr$ as its enveloping algebra
and $r$ as its unit element \cite[Theorem 4.10]{FSynap}. The orthomodular
lattice of projections in $rAr$ is the sublattice $P[0,r]:=\{p\in P:p
\leq r\}$ of $P$, and the orthocomplement in $P[0,r]$ of $p\in P[0,r]$ is
$p\sp{\perp\sb{r}}:=p\sp{\perp}\wedge r=r-p$. If $t$ is a symmetry in the
synaptic algebra $rAr$, then $t$ is a partial symmetry in $A$ and its
canonical extension to a symmetry in $A$ is $u:=t+r\sp{\perp}$. Thus, if
$p,q\in P[0,r]$ and if $p$ and $q$ are exchanged by a symmetry $t$ in
$rAr$, then $p$ and $q$ are exchanged by a symmetry $u$ in $A$. Let $a\in
rAr$. Then, if $0\leq a$, it follows that $a\sp{1/2}\in rAr$. Moreover,
$|a|\in rAr$, $a\dg\in rAr$, and $a\dg$ is the carrier of $a$ as calculated
in $rAr$.

The well-known \emph{Peirce decomposition} of $a\in A$ with respect to
$p\in P$, namely
\[
a=pap+pap\sp{\perp}+p\sp{\perp}ap+p\sp{\perp}ap\sp{\perp},
\]
is easily proved by direct calculation using the fact that $p\sp{\perp}
=1-p$. We note that $pap\sp{\perp}$ and $p\sp{\perp}ap$ belong to the
enveloping algebra $R$, but not necessarily to $A$; however $pap,\
pap\sp{\perp}+p\sp{\perp}ap,\  p\sp{\perp}ap\sp{\perp}\in A$.

As suggested by the following example, in our work the Peirce
decomposition will serve as a substitute for the operator matrix
formulas appearing in \cite{Guide, Halmos}.

\begin{example} \label{ex:matrices1}
{\rm To motivate and illustrate our subsequent work, we consider the case
in which ${\mathcal H}$ is a two-dimensional real Hilbert space, $A$ is
the rank 2 synaptic algebra of all self-adjoint linear operators on
${\mathcal H}$, $a\in A$, and $p\in P\setminus\{0,1\}$. Then we can
choose an orthonormal basis for ${\mathcal H}$ such that $a$, $p$, and
$p\sp{\perp}$, are represented by the matrices
\[
a=\left[\begin{array}{lr}\alpha & \gamma\\ \gamma & \beta\end{array}\right],
 \ p=\left[\begin{array}{lr}1 & 0\\ 0  & 0\end{array}\right], \ p\sp{\perp}=
 \left[\begin{array}{lr}0 & 0\\ 0  & 1\end{array}\right],
\]
where $\alpha, \beta, \gamma\in\reals$. Then, in the Peirce decomposition,
\[
pap=\left[\begin{array}{lr}\alpha & 0\\ 0 & 0\end{array}\right],
 \ pap\sp{\perp}=\left[\begin{array}{lr}0 & \gamma\\ 0  & 0\end{array}
 \right], \ p\sp{\perp}ap=\left[\begin{array}{lr}0 & 0\\ \gamma  & 0
 \end{array}\right],\ p\sp{\perp}ap\sp{\perp}=\left[\begin{array}{lr}
 0 & 0\\ 0  & \beta\end{array}\right].
\]}
\end{example}

\begin{definition} \label{df:diag&off-diag}
With the example above in mind, we shall refer to $pap+p\sp{\perp}ap
\sp{\perp}$ as the \emph{diagonal part} and to $pap\sp{\perp}+p\sp{\perp}
ap$ as the \emph{off-diagonal part} of $a\in A$ with respect to $p\in P$.
\end{definition}

\begin{lemma} \label{lm:diag,offdiag,zero}
Let $a\in A$ and $p\in P$. Than{\rm: (i)} If $0\leq a$, then $a=0$ iff
the diagonal part of $a$ with respect to $p$ is zero. {\rm(ii)} $aCp$
iff the off-diagonal part of $a$ with respect to $p$ is zero.
\end{lemma}

\begin{proof}
(i) Assume that $pap+p\sp{\perp}ap\sp{\perp}=0$. Then, since $0\leq a$,
we have $pap=0$ and $p\sp{\perp}ap\sp{\perp}=0$, so $(pap)\dg=0$
and $(p\sp{\perp}ap\sp{\perp})\dg=0$, and it follows from Lemma
\ref{lm:carofpqp} (i) that $p\wedge(p\sp{\perp}\vee a\dg)=0$ and
$p\sp{\perp}\wedge(p\vee a\dg)=0$. From $p\wedge(p\sp{\perp}\vee a
\dg)=0$, we have $p\sp{\perp}\vee(p\wedge(a\dg)\sp{\perp})=p\sp{\perp}
\oplus(p\wedge(a\dg)\sp{\perp})=1=p\sp{\perp}\oplus p$, whence $p
\wedge(a\dg)\sp{\perp}=p$ by cancellation, and we conclude that
$p\leq(a\dg)\sp{\perp}$, i.e., $a\dg\leq p\sp{\perp}$. Likewise,
from $p\sp{\perp}\wedge(p\vee a\dg)=0$, we deduce that $a\dg\leq p$,
and it follows that $a\dg=0$, and therefore $a=0$. The converse is
obvious.

(ii) Assume that $pap\sp{\perp}+p\sp{\perp}ap=0$. Then $a=pap+p
\sp{\perp}ap\sp{\perp}$, whence $ap=pa=pap$. The converse is obvious.
\end{proof}

\section{Two projections---basics} \label{sc:2projcom}

\begin{assumption} \label{as:pandq}
In what follows, we assume that $p$ and $q$ are arbitrary but fixed
projections in the OML $P$.
\end{assumption}

Naturally, the orthocomplements $p\sp{\perp}$ and $q\sp{\perp}$ will
have important roles to play in our subsequent study of $p$, $q$, and
their mutual interaction.  Accordingly, we shall be focusing our
attention on the four projections
\[
p,\,q,\,p\sp{\perp},\,q\sp{\perp}
\]
and certain ``lattice polynomials" in $p,\,q,\,p\sp{\perp}$,
and $q\sp{\perp}$, i.e., projections constructed from these four
using lattice meet, join, and orthocomplementation in $P$.

In the following definition, we extend the notion of the projection
$r\sb{p}$ in Theorem \ref{th:rpandcommutativity} to the projections
$r\sb{p\sp{\perp}}$, $r\sb{q}$, and $r\sb{q\sp{\perp}}$. The
alternative formulation as a product in each part of the definition
is justified by the fact that the projections in each threefold meet
commute with one another.

\begin{definition} \label{df:rsbp,etc}
\
\begin{enumerate}
\item[(1)] $r\sb{p}:=p\wedge(p\sp{\perp}\vee q)\wedge(p\sp{\perp}\vee
 q\sp{\perp})=p(p\sp{\perp}\vee q)(p\sp{\perp}\vee q\sp{\perp})$.
\item[(2)] $r\sb{p\sp{\perp}}:=p\sp{\perp}\wedge(p\vee q)\wedge(p\vee
 q\sp{\perp})=p\sp{\perp}(p\vee q)(p\vee q\sp{\perp})$.
\item[(3)] $r\sb{q}:=q\wedge(p\vee q\sp{\perp})\wedge(p\sp{\perp}\vee
 q\sp{\perp})=q(p\vee q\sp{\perp})(p\sp{\perp}\vee q\sp{\perp})$.
\item[(4)] $r\sb{q\sp{\perp}}:=q\sp{\perp}\wedge(p\vee q)\wedge(p\sp
 {\perp}\vee q)=q\sp{\perp}(p\vee q)(p\sp{\perp}\vee q)$.
\end{enumerate}
\end{definition}

By  Theorem \ref{th:rpandcommutativity} and symmetry, we have next two
theorems.

\begin{theorem} \label{th:equivcomconds}
The following conditions are mutually equivalent{\rm:}
\begin{enumerate}
\item At least one of the conditions $r\sb{p}=0$, $r\sb{p\sp{\perp}}=0$,
 $r\sb{q}=0$, or $r\sb{q\sp{\perp}}=0$ holds.
\item $r\sb{p}=r\sb{p\sp{\perp}}=r\sb{q}=r\sb{q\sp{\perp}}=0$.
\item At least one of the conditions $pCq$, $pCq\sp{\perp}$, $p\sp{\perp}
 Cq$, or $p\sp{\perp}Cq\sp{\perp}$ holds.
\item $pCq$, $pCq\sp{\perp}$, $p\sp{\perp}Cq$, and $p\sp{\perp}Cq\sp{\perp}$.
\item $[p,q]=0$.
\end{enumerate}
\end{theorem}

\begin{theorem} \label{th:rp,etc}
\
\begin{enumerate}
\item $p=(p\wedge q)\oplus(p\wedge q\sp{\perp})\oplus r\sb{p}$.
\item $p\sp{\perp}=(p\sp{\perp}\wedge q)\oplus(p\sp{\perp}\wedge q
 \sp{\perp})\oplus r\sb{p\sp{\perp}}$.
\item $q=(p\wedge q)\oplus(p\sp{\perp}\wedge q)\oplus r\sb{q}$.
\item $q\sp{\perp}=(p\wedge q\sp{\perp})\oplus(p\sp{\perp}\wedge
 q\sp{\perp})\oplus r\sb{q\sp{\perp}}$
\end{enumerate}
\end{theorem}

\begin{corollary} \label{co:commute}
$pCq$ iff $r\sb{p}=r\sb{q}=0$ iff $r\sb{p}Cr\sb{q}$.
\end{corollary}

\begin{proof}
That $pCq\Rightarrow r\sb{p}=r\sb{q}=0$ follows from Theorem
\ref{th:equivcomconds} and obviously $r\sb{p}=r\sb{q}=0\Rightarrow
r\sb{p}Cr\sb{q}$. Suppose that $r\sb{p}Cr\sb{q}$. Then in parts
(i) and (iii) of Theorem \ref{th:rp,etc}, every summand in the
orthogonal decomposition of $p$ commutes with every summand in
the orthogonal decomposition of $q$, whence $pCq$.
\end{proof}

By parts (ii) and (iii) of the following theorem, the unit element $1
\in A$ is the orthosum, hence also the supremum, (in two different
ways) of six pairwise orthogonal projections determined by $p$ and $q$.

\begin{theorem} \label{th:unitisorthosum}
\
\begin{enumerate}
\item $r\sb{p}\perp r\sb{p\sp{\perp}}$ and $r\sb{q}\perp r\sb
 {q\sp{\perp}}$.
\item $1=(p\wedge q)\oplus(p\wedge q\sp{\perp})\oplus(p\sp{\perp}\wedge q)
 \oplus(p\sp{\perp}\wedge q\sp{\perp})\oplus r\sb{p}\oplus r\sb{p\sp{\perp}}
 =[p,q]\sp{\perp}\oplus r\sb{p}\oplus r\sb{p\sp{\perp}}$.
\item $1=(p\wedge q)\oplus(p\wedge q\sp{\perp})\oplus(p\sp{\perp}\wedge q)
 \oplus(p\sp{\perp}\wedge q\sp{\perp})\oplus r\sb{q}\oplus r\sb{q\sp{\perp}}
 =[p,q]\sp{\perp}\oplus r\sb{q}\oplus r\sb{q\sp{\perp}}$.
\item $r\sb{p}\oplus r\sb{p\sp{\perp}}=r\sb{q}\oplus r\sb{q\sp{\perp}}=[p,q]$.
\end{enumerate}
\end{theorem}

\begin{proof}

Part (i) follows from obvious facts that $r\sb{p}\leq p$, $r\sb{p\sp{\perp}}
\leq p\sp{\perp}$, $r\sb{q}\leq q$, and $r\sb{q\sp{\perp}}\leq q\sp{\perp}$.

Part (ii) is a consequence of parts (i) and (ii) of Theorem \ref{th:rp,etc},
the fact that $1=p\oplus p\sp{\perp}$, and Definition \ref{df:commutator}.
Likewise, part (iii) follows from parts (iii) and (iv) of Theorem \ref{th:rp,etc}
and $1=q\oplus q\sp{\perp}$.

Part (iv) follows from (ii) and (iii).
\end{proof}

In the sixfold orthogonal decompositions of the unit $1$ in parts (ii)
and (iii) of Theorem \ref{th:unitisorthosum}, we are inclined to agree
with Halmos \cite[p. 381]{Halmos} that the first four projections
$p\wedge q$, $p\wedge q\sp{\perp}$, $p\sp{\perp}\wedge q$, and $p
\sp{\perp}\wedge q\sp{\perp}$ are ``thoroughly uninteresting." What is
interesting, is what Halmos refers to as ``the rest," namely the
projections $r\sb{p}\oplus r\sb{p\sp{\perp}}$ and $r\sb{q}\oplus r
\sb{q\sp{\perp}}$. Accordingly, in what follows, we pay special attention
to the projections $r\sb{p}$, $r\sb{p\sp{\perp}}$, $r\sb{q}$, $r\sb
{q\sp{\perp}}$, and $r\sb{p}\oplus r\sb{p\sp{\perp}}=r\sb{q}\oplus
r\sb{q\sp{\perp}}=[p,q]$.

\begin{definition} \label{df:r}
\
\begin{enumerate}
\item[(1)] $r:=r\sb{p}\oplus r\sb{p\sp{\perp}}=r\sb{q}\oplus r\sb{q\sp{\perp}}
 =r\sb{p} \vee r\sb{p\sp{\perp}}=r\sb{q}\vee r\sb{q\sp{\perp}}=[p,q]$.
\item[(2)] We call the synaptic algebra $rAr$ the \emph{commutator algebra}
 of $p$ and $q$.
\end{enumerate}
\end{definition}

\begin{theorem} \label{th:p,q,andr}
\
\begin{enumerate}
\item $pCq$ iff $r=0$.
\item $pr\sb{p}=r\sb{p}p=r\sb{p}$, $pr\sb{p\sp{\perp}}=r\sb{p\sp{\perp}}
 p=0$, and $pr=rp=p\wedge r=r\sb{p}$.
\item $p\sp{\perp}r\sb{p}=r\sb{p}p\sp{\perp}=0$, $p\sp{\perp}r\sb{p\sp
 {\perp}}=r\sb{p\sp{\perp}}p\sp{\perp}=r\sb{p\sp{\perp}}$, and $p\sp{\perp}r
 =rp\sp{\perp}=p\sp{\perp}\wedge r=r\sb{p\sp{\perp}}$.
\item $qr\sb{q}=r\sb{q}q=r\sb{q}$, $qr\sb{q\sp{\perp}}=r\sb{q\sp{\perp}}q
 =0$, and $qr=rq=q\wedge r=r\sb{q}$.
\item $q\sp{\perp}r\sb{q}=r\sb{q}q\sp{\perp}=0$, $q\sp{\perp}r\sb{q\sp{\perp}}
 =r\sb{q\sp{\perp}}q\sp{\perp}=r\sb{q\sp{\perp}}$, and $q\sp{\perp}r=rq
 \sp{\perp}=q\sp{\perp}\wedge r=r\sb{q\sp{\perp}}$.
\item $p,p\sp{\perp},q,$ and $q\sp{\perp}$ commute with $r=r\sb{p}\oplus r\sb
 {p\sp{\perp}}=r\sb{q}\oplus r\sb{q\sp{\perp}}=[p,q]$.
\item $r\sb{p}$, $r\sb{p\sp{\perp}}$, $r\sb{p}\!\sp{\perp}$, $r\sb{q}$,
 $r\sb{q\sp{\perp}}$,and $r\sb{q}\!\sp{\perp}$ commute with $r$.
 \item $r\sb{p}=p\wedge r=pr=rp$, $r\sb{p\sp{\perp}}=p\sp{\perp}\wedge r=
 p\sp{\perp}r=rp\sp{\perp}$, $r\sb{q}=q\wedge r=qr=rq$, and $r\sb{q\sp{\perp}}
 =q\sp{\perp}\wedge r=q\sp{\perp}r=rq\sp{\perp}$.
\end{enumerate}
\end{theorem}

\begin{proof}
By Definition \ref{df:r} and Theorem \ref{th:equivcomconds} (v), we have
(i). Part (ii) follows from the facts that $r\sb{p}\leq p$, $r\sb{p\sp{\perp}}
\leq p\sp{\perp}$, and $r=r\sb{p}+r\sb{p\sp{\perp}}$. Similar arguments
prove (iii), (iv), and (v). Parts (vi) and (vii) follow from (ii)--(v).

Since $r=(p\vee q)\wedge(p\vee q\sp{\perp})\wedge(p\sp{\perp}\vee q)
\wedge(p\sp{\perp}\vee q\sp{\perp})$, it follows that $p\wedge r=p\wedge
(p\sp{\perp}\vee q)\wedge(p\sp{\perp}\vee q\sp{\perp})=r\sb{p}$, and
$p\wedge r=pr=rp$ because $pCr$. The remaining equalities in (viii) are
proved similarly.
\end{proof}

\section{Sine and cosine effect elements} \label{sc:sinecosine}

\begin{example} \label{ex:matrices2}
{\rm Again we consider the rank 2 synaptic algebra in Example \ref
{ex:matrices1}, this time having a look at the situation of present
interest in which $p,q\in P$. Assuming that $p,q\not=0,1$, and
$p\not=q,q\sp{\perp}$, we can choose an orthonormal basis for
${\mathcal H}$ such that $p$ diagonalizes and $p$, $q$, $p
\sp{\perp}$, and $q\sp{\perp}$ are represented by the matrices
\[
p=\left[\begin{array}{lr}1 & 0\\ 0 & 0\end{array}\right],\ q=\left
 [\begin{array}{lr}\cos\sp{2}\theta & \cos\,\theta \sin\,\theta\\
 \cos\,\theta \sin\,\theta  & \sin\sp{2}\theta\end{array}\right],
\]
\[
p\sp{\perp}=\left[\begin{array}{lr}0 & 0\\ 0 & 1\end{array}\right],\
q\sp{\perp}=\left[\begin{array}{lr}\sin\sp{2}\theta & -\cos\,\theta \sin
\,\theta\\ -\cos\,\theta \sin\,\theta & \cos\sp{2}\theta\end{array}\right],
\]
where $0<\theta<\frac{\pi}{2}$ is the positive acute angle between the
one-dimensional subspaces upon which $p$ and $q$ project (see \cite
[p. 384]{Halmos}).}
\end{example}

In Example \ref{ex:matrices2}, we have
\[
pqp+p\sp{\perp}q\sp{\perp}p\sp{\perp}=(\cos\sp{2}\theta)I\text{\ and\ }
 pq\sp{\perp}p+p\sp{\perp}qp\sp{\perp}=(\sin\sp{2}\theta)I,
\]
respectively, where $I$ is the identity matrix. This suggests the following
definition.

\begin{definition} \label{df:cossineffects}
As $0\leq q,\,q\sp{\perp}$, we have $0\leq pqp,\, p\sp{\perp}q\sp{\perp}
p\sp{\perp},\, pq\sp{\perp}p,\, p\sp{\perp}qp\sp{\perp}$, so $0\leq pqp+
p\sp{\perp}q\sp{\perp}p\sp{\perp}$ and $0\leq pq\sp{\perp}p+p\sp{\perp}
qp\sp{\perp}$, whence we define
\[
 c:=(pqp+p\sp{\perp}q\sp{\perp}p\sp{\perp})\sp{1/2}\text{\ and\ }
 s:=(pq\sp{\perp}p+p\sp{\perp}qp\sp{\perp})\sp{1/2}.
\]
We refer to $c$ as the \emph{cosine effect} and to $s$ as the
\emph{sine effect} for the projection $q$ with respect to $p$.
\end{definition}

\noindent Recall that an element $e\in A$ is called an \emph{effect}
iff $0\leq e\leq 1$. In part (vi) of the next theorem, we show that $c$
and $s$ are, in fact, effects in $A$.

\begin{theorem} \label{th:sceffects}
\
\begin{enumerate}
\item $c\sp{2}=pqp+p\sp{\perp}q\sp{\perp}p\sp{\perp}=1-(p-q)\sp{2}
 =(p-q\sp{\perp})\sp{2}=(p+q-1)\sp{2}$.
\item $s\sp{2}=pq\sp{\perp}p+p\sp{\perp}qp\sp{\perp}=(p-q)\sp{2}$.
\item $pc\sp{2}=pqp=c\sp{2}p,\ qc\sp{2}=qpq=c\sp{2}q,\ ps\sp{2}=
 pq\sp{\perp}p=s\sp{2}p,$\newline $qs\sp{2}=qp\sp{\perp}q=s\sp{2}q$,
 and $s\sp{2}p\sp{\perp}=p\sp{\perp}qp\sp{\perp}=p\sp{\perp}s\sp{2}$.
\item $c=|p-q\sp{\perp}|$ and $s=|p-q|$.
\item $c\sp{2}+s\sp{2}=1$.
\item $0\leq c\sp{2},\,s\sp{2},\,c,\,s\leq 1$. Also, $c\sp{2}\leq c$, and
 $s\sp{2}\leq s$.
\item $C(c)=C(c\sp{2})=C(s\sp{2})=C(s)$.
\item $cCp$, $cCq$, $cCr$, $sCp$, $sCq$, $sCr$, and  $cCs$.
\item $C(p)\cap C(q)\subseteq C(c)=C(s)$.
\end{enumerate}
\end{theorem}

\begin{proof}
By direct calculation using Definition \ref{df:cossineffects} and the facts
that $p\sp{\perp}=1-p$ and $q\sp{\perp}=1-q$, we have $c\sp{2}=pqp+p\sp{\perp}
q\sp{\perp}p\sp{\perp}=1-p+pq+qp-q=1-(p-q)\sp{2}$. Also, $(p-q\sp{\perp})
\sp{2}=(p+q-1)\sp{2}=1-p+pq+qp-q$, and (i) follows. Similarly, $s\sp{2}=pq
\sp{\perp}p+p\sp{\perp}qp\sp{\perp}=p-pq-qp+q=(p-q)\sp{2}$, proving (ii).

Part (iii) follows by direct calculation using the facts that $c\sp{2}
=pqp+p\sp{\perp}q\sp{\perp}p\sp{\perp}=1-p+pq+qp-q$ and $s\sp{2}=pq
\sp{\perp}p+p\sp{\perp}qp\sp{\perp}=p-pq-qp+q$.

Part (iv) follows from (i), (ii), and the facts that $0\leq c,s$.

Part (v) follows from $c\sp{2}=1-(q-p)\sp{2}=1-s\sp{2}$. Obviously,
$0\leq c\sp{2},\, s\sp{2},\, c,\, s$.  Also, $c\sp{2}, s\sp{2}\leq
c\sp{2}+s\sp{2}=1$. By \cite[Corollary 3.4]{FSynap}, the facts
that $0\leq s$, $0\leq 1$, $sC1$, and $s\sp{2}\leq 1=1\sp{2}$ imply
that $s\leq 1$. Likewise, $c\leq 1$, so $0\leq c,\,s,\,c\sp{2},\,s
\sp{2}\leq 1$. That $c\sp{2}\leq c$ and $s\sp{2}\leq s$ then follow
from \cite[Lemma 2.5 (i)]{FSynap}, and (vi) is proved.

Since $0\leq c,s$ and $c\sp{2}=1-s\sp{2}$, it follows that $C(c)=
C((c\sp{2})\sp{1/2})=C(c\sp{2})=C(1-s\sp{2})=C(s\sp{2})=C((s\sp{2})
\sp{1/2})=C(s)$, proving (vii).

By (iii) and (vii), $p,q\in C(c)=C(s)$, whence $cCr$ and $sCr$. Also,
$c\in C(c)=C(s)$, completing the proof of (viii).

By parts (iv) and (vii) above, $a\in C(p)\cap C(q)\Rightarrow a\in
C(p-q\sp{\perp})\Rightarrow a\in C(|p-q\sp{\perp}|)\Rightarrow
a\in C(c)=C(s)$, proving (ix).
\end{proof}

The following theorem concerns the carriers $c\dg$ and $s\dg$ of the
cosine and sine effects $c$ and $s$.

\begin{theorem} \label{th:carofc&s}
\
\begin{enumerate}
\item $c\dg Cp$, $c\dg Cq$, $c\dg Cs$, $c\dg Cr$, $s\dg Cp$, $s\dg Cq$, $s\dg Cr$,
 $s\dg Cc$, and $c\dg C s\dg$.
\item $c\dg=(p\vee q\sp{\perp})\wedge(p\sp{\perp}\vee q)$ and
 $s\dg=(p\vee q)\wedge(p\sp{\perp}\vee q\sp{\perp}).$
\item $(cs)\dg=c\dg s\dg= c\dg\wedge s\dg=(p\vee q)\wedge(p\vee q\sp{\perp})
 \wedge(p\sp{\perp}\vee q)\wedge(p\sp{\perp}\vee q\sp{\perp})=r=[p,q]$.

\item $c\sp{2}{s\dg}\sp{\perp}={s\dg}\sp{\perp}c\sp{2}={s\dg}\sp{\perp}$,
 whence ${s\dg}\sp{\perp}\leq c\sp{2}\leq c$.
\item $s\sp{2}{c\dg}\sp{\perp}={c\dg}\sp{\perp}s\sp{2}={c\dg}\sp{\perp}$,
 whence ${c\dg}\sp{\perp}\leq s\sp{2}\leq s$.

\end{enumerate}
\end{theorem}

\begin{proof}

(i) Part (i) follows from Theorem \ref{th:sceffects} (viii), and the
facts that $c\dg\in CC(c)$ and $s\dg\in CC(s)$.

(ii) Since $0\leq pqp,\ p\sp{\perp}q\sp{\perp}p\sp{\perp}$, we infer from
Theorem \ref{lm:carofpqp} and Lemma \ref{lm:carrierofsum} that
\setcounter{equation}{0}
\[
c\dg=(c\sp{2})\dg=(pqp+p\sp{\perp}q\sp{\perp}p\sp{\perp})\dg
 =(pqp)\dg\vee(p\sp{\perp}q\sp{\perp}p\sp{\perp})\dg
\]
\begin{equation} \label{eq:car01}
=(p\wedge(p\sp{\perp}\vee q))\vee(p\sp{\perp}\wedge(p\vee q\sp{\perp}))
=(p\wedge(p\sp{\perp}\vee q))\vee w,
\end{equation}
where $w:=p\sp{\perp}\wedge(p\vee q\sp{\perp})$. Now $pC(p\sp{\perp}\vee
q)$ and $pCw$, whence
\begin{equation} \label{eq:car02}
(p\wedge(p\sp{\perp}\vee q))\vee w=(p\vee w)\wedge(p\sp{\perp}\vee q\vee w).
\end{equation}
But $pCp\sp{\perp}$ and $pC(p\vee q\sp{\perp})$, whence
\begin{equation} \label{eq:car03}
p\vee w=p\vee(p\sp{\perp}\wedge(p\vee q\sp{\perp}))=(p\vee p\sp{\perp})
\wedge(p\vee p\vee q\sp{\perp})=p\vee q\sp{\perp}.
\end{equation}
Furthermore, since $w\leq p\sp{\perp}$,
\begin{equation} \label{eq:car04}
p\sp{\perp}\vee q\vee w=p\sp{\perp}\vee q.
\end{equation}
By Equations (\ref{eq:car03}) and (\ref{eq:car04}),
\[
(p\vee w)\wedge(p\sp{\perp}\vee q\vee w)=(p\vee q\sp{\perp})
 \wedge(p\sp{\perp}\vee q),
\]
whence by Equations (\ref{eq:car02}) and (\ref{eq:car01}),
$c\dg=(p\vee q\sp{\perp})\wedge(p\sp{\perp}\vee q)$.

Similarly,
\[
s\dg=(s\sp{2})\dg=((pq\sp{\perp}p+p\sp{\perp}qp\sp{\perp}))\dg,
\]
and replacing $q$ by $q\sp{\perp}$ in the calculations above, we
find that $s\dg=(p\vee q)\wedge(p\sp{\perp}\vee q\sp{\perp})$,
completing the proof of (ii).

(iii)  Part (iii) follows from Theorem (ii), \ref{th:sceffects} (viii),
Lemma \ref{lm:carrierofprod}, and Definitions \ref{df:commutator} and
\ref{df:r}.

(iv) Since ${s\dg}\sp{\perp}c\sp{2}={s\dg}\sp{\perp}(1-s\sp{2})=
{s\dg}\sp{\perp}-0={s\dg}\sp{\perp}$, we have ${s\dg}\sp{\perp}\leq c
\sp{2}\leq c$.

(iv) Since ${c\dg}\sp{\perp}s\sp{2}={c\dg}\sp{\perp}(1-c\sp{2})=
{c\dg}\sp{\perp}-0={c\dg}\sp{\perp}$, whence ${c\dg}\sp{\perp}
\leq s\sp{2}\leq s$.
\end{proof}

Using Theorem \ref{th:carofc&s}, we obtain formulas for Halmos'
four ``thoroughly uninteresting" projections in terms of $p$, $q$,
$c$, and $s$ as follows.

\begin{corollary} \label{lm:uninteresting}
\
\begin{enumerate}
\item ${s\dg}\sp{\perp}p=p{s\dg}\sp{\perp}={s\dg}\sp{\perp}q=q{s\dg}\sp{\perp}
 ={s\dg}\sp{\perp}\wedge p={s\dg}\sp{\perp}\wedge q=p\wedge q$.
\item ${c\dg}\sp{\perp}p=p{c\dg}\sp{\perp}={c\dg}\sp{\perp}q\sp{\perp}
=q\sp{\perp}{c\dg}\sp{\perp}={c\dg}\sp{\perp}\wedge p={c\dg}\sp{\perp}
\wedge q\sp{\perp}=p\wedge q\sp{\perp}$.
\item ${c\dg}\sp{\perp}p\sp{\perp}=p\sp{\perp}{c\dg}\sp{\perp}={c\dg}\sp{\perp}q
 =q{c\dg}\sp{\perp}={c\dg}\sp{\perp}\wedge p\sp{\perp}={c\dg}\sp{\perp}\wedge q
 =p\sp{\perp}\wedge q$.
\item ${s\dg}\sp{\perp}p\sp{\perp}=p\sp{\perp}{s\dg}\sp{\perp}={s\dg}\sp{\perp}
 q\sp{\perp}=q\sp{\perp}{s\dg}\sp{\perp}={s\dg}\sp{\perp}\wedge p\sp{\perp}=
 {s\dg}\sp{\perp}\wedge q\sp{\perp}=p\sp{\perp}\wedge q\sp{\perp}$.
\end{enumerate}
\end{corollary}

\begin{proof}
As a consequence of Theorem \ref{th:carofc&s} (i), the projections
${c\dg}\sp{\perp}$ and ${s\dg}\sp{\perp}$ commute with both $p$ and $q$.
Also, by Theorem \ref{th:carofc&s} (ii) and De\,Morgan,
\[
{c\dg}\sp{\perp}=(p\sp{\perp}\wedge q)\vee(p\wedge q\sp{\perp})
 \text{\ and\ } {s\dg}\sp{\perp}=(p\sp{\perp}\wedge q\sp{\perp})
 \vee(p\wedge q).
\]
We prove (i). Proofs of (ii), (iii), and (iv) are similar. We have
${s\dg}\sp{\perp}p=p{s\dg}\sp{\perp}=p\wedge[(p\sp{\perp}\wedge q
\sp{\perp})\vee(p\wedge q)]=(p\wedge p\sp{\perp}\wedge q\sp{\perp})
\vee(p\wedge p\wedge q)=p\wedge q$. Likewise, ${s\dg}\sp{\perp}q=
q{s\dg}\sp{\perp}=q\wedge[(p\sp{\perp}\wedge q\sp{\perp})\vee(p
\wedge q)]=(q\wedge p\sp{\perp}\wedge q\sp{\perp})\vee(q\wedge p
\wedge q)=p\wedge q$.
\end{proof}

\section{A general CS-decomposition theorem}
\label{sc:genCS}

We devote this section to a proof of a general CS-decomposition
theorem that does not require the projections $p$ and $q$ to be
in generic position (Theorem \ref{th:genCSdecomp} below).

By Theorem \ref{th:sceffects} (iii), we have the following result.
\begin{lemma} \label{lm:csPeirce}
The Peirce decomposition of $q$ with respect to $p$ takes the form
\[
q=pqp+pqp\sp{\perp}+p\sp{\perp}qp+p\sp{\perp}qp\sp{\perp}=c\sp{2}p
+pqp\sp{\perp}+p\sp{\perp}qp+s\sp{2}p\sp{\perp}.
\]
\end{lemma}

\noindent Thus, for the diagonal part of the Peirce decomposition
of $q$ with respect to $p$, we have
\[
pqp+p\sp{\perp}qp\sp{\perp}=c\sp{2}p+s\sp{2}p\sp{\perp},
\]
which is perfectly consistent with Halmos' Theorem 2 in \cite{Halmos},
often called Halmos' \emph{two projections theorem} or Halmos' \emph
{CS-decomposition theorem}. However, for full compliance with Halmos'
theorem, we have to find a suitable formula, \emph{in terms of the
product} $cs$, for the off-diagonal part $pqp\sp{\perp}+p\sp{\perp}qp$
of the decomposition. (Note that Halmos' theorem was proved under the
additional hypothesis that the projections involved are in generic
position---see Section \ref{sc:Generic} below.) In this connection,
the next theorem has a role to play.

\begin{theorem} \label{th:squarecs}
\
\begin{enumerate}
\item $c\sp{2}s\sp{2}=pqp+qpq-pqpq-qpqp=p(qp\sp{\perp}q)p+p\sp{\perp}
(qpq)p\sp{\perp}\newline\hspace*{.27 in} =(pqp\sp{\perp}+p\sp{\perp}qp)\sp{2}$.
\item $cs=|pqp\sp{\perp}+p\sp{\perp}qp|$.
\end{enumerate}
\end{theorem}

\begin{proof}\

(i) We have $c\sp{2}=1-(p-q)\sp{2}=1-p+pq+qp-q$ and $s\sp{2}=(p-q)\sp{2}=
p-pq-qp+q$, and it follows by direct calculation that $c\sp{2}s\sp{2}=
pqp+qpq-pqpq-qpqp$. Also, direct calculations using the fact that
$p\sp{\perp}=1-p$ yield $pqp+qpq-pqpq-qpqp=p(qp\sp{\perp}q)p+p\sp{\perp}
(qpq)p\sp{\perp}=(pqp\sp{\perp}+p\sp{\perp}qp)\sp{2}$.

(ii) As $cs=sc$, we have $(cs)\sp{2}=c\sp{2}s\sp{2}$. Also, $0\leq c,s$ and
$cs=sc$, so $0\leq cs$ by \cite[Lemma 1.5]{FSynap}, and (ii) then follows.
\end{proof}

\begin{definition} \label{df:symsuvk}
As per Theorem \ref{th:sceffects} (iv) and Theorem \ref{th:squarecs}
(ii), we define symmetries $u$, $v$, and $k$ in $A$ by
polar decomposition of $p-q\sp{\perp}$, $p-q$, and $pqp\sp{\perp}+
p\sp{\perp}qp$, respectively, as
follows:
\begin{enumerate}
\item[(1)] $p-q\sp{\perp}=p+q-1=cu=uc$, where $u\in CC(p-q\sp{\perp})$.
\item[(2)] $p-q=sv=vs$, where $v\in CC(p-q)$.
\item[(3)] $pqp\sp{\perp}+p\sp{\perp}qp=csk=kcs$, where $k\in
 CC(pqp\sp{\perp}+p\sp{\perp}qp)=p-ps-sp+s$.
\end{enumerate}
\end{definition}

\begin{lemma} \label{lm:uvkCsc}
The symmetries $u$, $v$, and $k$ commute with both $s$ and $c$.
\end{lemma}

\begin{proof}
We already know that $uCc$; hence, since $C(c)=C(s)$ (Theorem
\ref{th:sceffects} (vii)), we have $uCs$. Similarly, we already
know that $vCs$, and therefore $vCc$.

We have $cCp,\ cCq,\ sCp,\ sCq$, so $c,s\in C(pqp\sp{\perp}+
p\sp{\perp}qp)$. But $k\in CC(pqp\sp{\perp}+p\sp{\perp}qp)$, so
$kCc$ and $kCs$.
\end{proof}

\begin{lemma} \label{lm:uvProps}
\
\begin{enumerate}
\item $upqpu=qpq$ and $u(p\wedge(p\sp{\perp}\vee q))u=q\wedge
 (q\sp{\perp}\vee p)$.
\item $vpq\sp{\perp}pv=q\sp{\perp}pq\sp{\perp}$ and $v(p\wedge
 (p\sp{\perp}\vee q\sp{\perp})v=q\sp{\perp}\wedge(q\vee p)$.
\item $cs(pk+kp-k)=0$ and $r(pk+kp-k)=0$.
\end{enumerate}
\end{lemma}

\begin{proof}
Part (i) follows from Theorem \ref{th:6.5FPSynap} and part (ii)
follows from the same theorem upon replacing $q$ by $q\sp{\perp}$.

To prove (iii), we begin by noting that since $pCc$, $pCs$, and
$csk=pqp\sp{\perp}+p\sp{\perp}qp$, we have
\[
cspk=pcsk=p(pep\sp{\perp}+p\sp{\perp}ep)=pep\sp{\perp}.
\]
Moreover,
\[
cskp\sp{\perp}=(pep\sp{\perp}+p\sp{\perp}ep)p\sp{\perp}=pep\sp{\perp},
\]
and therefore
\[
cs(pk+kp-k)=cs(pk-kp\sp{\perp})=0,
\]
whence $r(pk+kp-k)=(cs)\dg(pk+kp-k)=0$ by Theorem \ref{th:carofc&s} (iii).
\end{proof}

Combining Lemma \ref{lm:csPeirce}, Definition \ref{df:symsuvk}, and Lemma
\ref{lm:uvkCsc}, we obtain the following generalized version of Halmos'
CS-decomposition theorem.

\begin{theorem}[Generalized CS-decomposition] \label{th:genCSdecomp}
\[
q=c\sp{2}p+csk+s\sp{2}p\sp{\perp},
\]
where $pqp=c\sp{2}p=pc\sp{2}$, $p\sp{\perp}qp\sp{\perp}=s\sp{2}p
\sp{\perp}=p\sp{\perp}s\sp{2}$, $pqp\sp{\perp}+p\sp{\perp}qp=csk$,
$k$ is a symmetry, $cCs$, $cCk$, $sCk$, and $k\in CC(pqp\sp{\perp}
+p\sp{\perp}qp)$.
\end{theorem}

Note that we do not have to assume that $p$ and $q$ are in generic
position in Theorem \ref{th:genCSdecomp}. However, at this point
in the development of our theory, we do not have much information
about the critical symmetry $k$ involved in the formula $pqp\sp{\perp}
+p\sp{\perp}qp=csk$ for the off-diagonal part of the Peirce decomposition
of $q$ with respect to $p$. Nevertheless, due to its generality, Theorem
\ref{th:genCSdecomp} can be useful.

\begin{corollary} \label{co:pCq}
Let $p,q\in P$ and let $c$ and $s$ be the cosine and sine effects for
$q$ with respect to $p$. Then the following conditions are mutually
equivalent{\rm:}
\begin{enumerate}
\item $pCq$.
\item The off-diagonal part of $q$ with respect to $p$ is zero, i.e.,
 $pqp\sp{\perp}+p\sp{\perp}qp=0$.
\item $q=c\sp{2}p+s\sp{2}p\sp{\perp}$.
\item $cs=0$
\item $c$ and $s$ are projections and $c\sp{\perp}=s$.
\item $c,s\in P$, $c\sp{\perp}=s$, and $q=cp+c\sp{\perp}p\sp{\perp}
 =s\sp{\perp}p+ sp\sp{\perp}=|p-s|$.
\item There exists a projection $t\in P$ such that $tCp$
 and $q=|p-t|$.
\end{enumerate}
\end{corollary}

\begin{proof}

The equivalence (i) $\Leftrightarrow$ (ii) is Lemma \ref
{lm:diag,offdiag,zero} (ii). By Theorem \ref{th:genCSdecomp},
$q=c\sp{2}p+csk+s\sp{2}p\sp{\perp}$ where $csk=pqp\sp{\perp}
+p\sp{\perp}qp$ and $k\sp{2}=1$, from which (ii) $\Leftrightarrow$
(iii) and (iii) $\Leftrightarrow$ (iv) both follow.

Now we claim that (iv) $\Rightarrow$ (v). Indeed, assume (iv). Then,
since $cCs$, $0=c\sp{2}s\sp{2}=c\sp{2}(1-c\sp{2})=c\sp{2}-(c\sp{2})
\sp{2}$, so $c\sp{2}\in P$, whence $c$ is a partial symmetry with
$0\leq c$. Thus, $c$ is a projection, and by a similar argument, so
is $s$; moreover, $c=c\sp{2}=1-s\sp{2}=1-s$, so $c\sp{\perp}=s$,
and we have (iv) $\Rightarrow$ (v). Conversely, if (v) holds, then
$cs=cc\sp{\perp}=0$, so (iv) holds, and we have (iv) $\Leftrightarrow$
(v). Thus we have the mutual equivalence of conditions (i) through (v).

Assume (i). Then (iii), hence also (v) holds, whence $c,s\in P$, $c=
s\sp{\perp}$, and $q=c\sp{2}p+s\sp{2}p\sp{\perp}=cp+c\sp{\perp}p
\sp{\perp}=s\sp{\perp}p+sp\sp{\perp}=p-sp-ps+s=(p-s)\sp{2}$. Therefore,
$q=q\sp{1/2}=|p-s|$. This proves that (i) $\Rightarrow$ (vi).

Obviously, with $t=s$, (vi) $\Rightarrow$ (vii), and it is clear that
(vii) $\Rightarrow$ (i).
\end{proof}

A generalized CS-decomposition for the projection $q\sp{\perp}$ with
respect to $p$ is easily obtained from Theorem \ref{th:genCSdecomp}.

\begin{corollary} \label{co:genCSdecomp}
$q\sp{\perp}=s\sp{2}p+cs(-k)+c\sp{2}p\sp{\perp}$.
\end{corollary}

\begin{proof}
We have $q\sp{\perp}=1-q=p+p\sp{\perp}-c\sp{2}p-csk-s\sp{2}p\sp{\perp}
=(1-c\sp{2})p+cs(-k)+(1-s\sp{2})p\sp{\perp}=s\sp{2}p+cs(-k)+c\sp{2}p\sp{\perp}$.
\end{proof}

\section{Generic position} \label{sc:Generic}

As an immediate consequence of Definitions \ref{df:generic} and
\ref{df:commutator}, we have the following.

\begin{lemma} \label{lm:genpos01}
$p$ and $q$ are in generic position iff $r=[p,q]=1$.
\end{lemma}

According to Theorem \ref{th:p,q,andr} (ii)--(v), the projections
$r\sb{p}, r\sb{p\sp{\perp}}, r\sb{q}$ and $r\sb{q\sp{\perp}}$ belong
to the lattice of projections $P[0,r]$ of the commutator algebra
$rAr$ of $p$ and $q$. In this section we are going to prove that
$r\sb{p}$ and $r\sb{q}$ are in generic position in $rAr$. We begin
with two preliminary lemmas, the first of which---an immediate
consequence of Theorem \ref{th:p,q,andr}---identifies $r\sb
{p\sp{\perp}}$ and $r\sb{q\sp{\perp}}$ as the orthocomplements
of $r\sb{p}$ and $r\sb{q}$ in $P[0,r]$.

\begin{lemma} \label{lm:meetwithr}
{\rm(i)} $r\sb{p}\sp{\perp\sb{r}}=r\sb{p}\!\sp{\perp}\wedge r=r\sb{p}
\!\sp{\perp}r=rr\sb{p}\!\sp{\perp}=r\sb{p\sp{\perp}}$. {\rm(ii)}
$r\sb{q}\sp{\perp\sb{r}}=r\sb{q}\!\sp{\perp}\wedge r=r\sb{q}\!\sp{\perp}r
=rr\sb{q}\!\sp{\perp}=r\sb{q\sp{\perp}}$.
\end{lemma}

\begin{lemma} \label{lm:disjointness}
$r\sb{p}\wedge r\sb{q}=r\sb{p}\wedge r\sb{q\sp{\perp}}=r\sb{p\sp{\perp}}
\wedge r\sb{q}=r\sb{p\sp{\perp}}\wedge r\sb{q\sp{\perp}}=0$.
\end{lemma}

\begin{proof}
We have $r\sb{p}=p\wedge(p\sp{\perp}\vee q)\wedge(p\sp{\perp}\vee q\sp
{\perp})\leq p\wedge(p\sp{\perp}\vee q\sp{\perp})$ and $r\sb{q}=q
\wedge(p\vee q\sp{\perp})\wedge(p\sp{\perp}\vee q\sp{\perp})\leq q$,
whence
\[
r\sb{p}\wedge r\sb{q}\leq p\wedge(p\sp{\perp}\vee q\sp{\perp})\wedge q
 =(p\wedge q)\wedge(p\wedge q)\sp{\perp}=0,
\]
whence $r\sb{p}\wedge r\sb{q}=0$. The remaining equalities follow by
symmetry.
\end{proof}

\begin{theorem} \label{th:rpandrqinGP}
The projections $r\sb{p}=pr=rp=r\wedge p\in P[0,r]$ and $r\sb{q}=qr=rq=r
\wedge q\in P[0,r]$ are in generic position in the commutator algebra $rAr$.
\end{theorem}

\begin{proof}
Combine Lemmas \ref{lm:meetwithr} and \ref{lm:disjointness}.
\end{proof}

\begin{corollary} \label{co:rpveerq}
$r\sb{p}\vee r\sb{q}=r\sb{p}\vee r\sb{q\sp{\perp}}=r\sb{p\sp{\perp}}
\vee r\sb{q}=r\sb{p\sp{\perp}}\vee r\sb{q\sp{\perp}}=[p,q]=r$.
\end{corollary}

In view of Theorem \ref{th:rpandrqinGP}, it seems natural to inquire about
the cosine and sine effects of $r\sb{q}$ with respect to $r\sb{p}$ as
calculated in $rAr$.

\begin{definition} \label{df:csinrAr}
\[
 c\sb{r}:=(r\sb{p}r\sb{q}r\sb{p}+r\sb{p\sp{\perp}}r\sb{q\sp{\perp}}
 r\sb{p\sp{\perp}})\sp{1/2}\text{\ and\ } s\sb{r}:=(r\sb{p}r\sb{q
 \sp{\perp}}r\sb{p}+r\sb{p\sp{\perp}}r\sb{q}r\sb{p\sp{\perp}})\sp{1/2}.
\]
\end{definition}

\begin{theorem} \label{th:csinrAr}
\
\begin{enumerate}
\item $c\sb{r}=cr=rc$ and $s\sb{r}=sr=rs$.
\item $c=c\sb{r}+|(p\wedge q)-(p\sp{\perp}\wedge q\sp{\perp})|$ and $s=
 s\sb{r}+|(p\wedge q\sp{\perp})-(p\sp{\perp}\wedge q)|$.
\end{enumerate}
\end{theorem}

\begin{proof}
(i) By Theorem \ref{th:sceffects}, we have $c\sb{r}=|r\sb{p}-r\sb{q\sp{\perp}}|$
and $s\sb{r}=|r\sb{p}-r\sb{q}|$. Thus, $c\sb{r}=|pr-q\sp{\perp}r|=|(p-q
\sp{\perp})r|$ and as $(p-q\sp{\perp})Cr$ and $0\leq r$, it follows that
$c\sb{r}=|p-q\sp{\perp}||r|=|p-q\sp{\perp}|r=cr=rc$. Similarly, $s\sb{r}=
|pr-qr|=|(p-q)r|=|p-q|r=sr=rs$.

(ii) As $r\sp{\perp}=(p\wedge q)\vee(p\wedge q\sp{\perp})\vee(p\sp{\perp}
\wedge q)\vee (p\sp{\perp}\wedge q\sp{\perp})=(p\wedge q)+(p\wedge q
\sp{\perp})+(p\sp{\perp}\wedge q)+(p\sp{\perp}\wedge q\sp{\perp})$, it follows
that $pr\sp{\perp}=(p\wedge q)+(p\wedge q\sp{\perp})$ and $q\sp{\perp}r
\sp{\perp}=(p\wedge q\sp{\perp})+(p\sp{\perp}\wedge q\sp{\perp})$. Thus,
$cr\sp{\perp}=|p-q\sp{\perp}|r\sp{\perp}=|pr\sp{\perp}-q\sp{\perp}r\sp{\perp}|
=|(p\wedge q)+(p\wedge q\sp{\perp})-(p\wedge q\sp{\perp})-(p\sp{\perp}
\wedge q\sp{\perp})|=|(p\wedge q)-(p\sp{\perp}\wedge q\sp{\perp})|$. Therefore,
$c=cr+cr\sp{\perp}=c\sb{r}+|(p\wedge q)-(p\sp{\perp}\wedge q\sp{\perp})|$.
A similar calculation yields $s=s\sb{r}+|(p\wedge q\sp{\perp})-(p\sp{\perp}
\wedge q)|$.
\end{proof}

\section{Dropping down to the commutator algebra} \label{sc:dropdown}

If $pCq$, then $r=0$ and by Theorems \ref{th:equivcomconds} and \ref
{th:rp,etc}, $p,q,p\sp{\perp},$ and $q\sp{\perp}$ can be expressed in
terms of Halmos' four ``thoroughly uninteresting" projections, essentially
concluding our study of $p$ and $q$.

Using parts (i) and (iii) of Theorem \ref{th:rp,etc}, part (ii) of
Theorem \ref{th:csinrAr}, and Halmos' four uninteresting projections,
we can translate properties of $r\sb{p}$, $r\sb{q},$ $c\sb{r},$ and
$s\sb{r}$ into properties of $p$, $q$, $c$, and $s$; hence these theorems
reduce the study of the two projections $p$ and $q$ in the synaptic algebra
$A$ to the study of the two projections $r\sb{p}$ and $r\sb{q}$, which by
Theorem \ref{th:rpandrqinGP} are in generic position in the commutator
algebra $rAr$ of $p$ and $q$. As we are going to assume that $p$ does not
commute with $q$, i.e., $r\not=0$, Corollary \ref{co:commute} will imply
that $r\sb{p}$ does not commute with $r\sb{q}$. Thus, we propose to drop down
from $A$ to the nondegenerate commutator algebra $rAr$ and focus on the study
of $r\sb{p}$ and $r\sb{q}$ in $rAr$. Consequently, to simplify notation, we
shall now replace the synaptic algebra $rAr$ by $A$ and replace $r\sb{p}$ and
$r\sb{q}$ by $p$ and $q$, respectively. Notice that this is exactly what was
done by B\"{o}ttcher and Spitkovsky \cite[p. 1414]{Guide}.

\begin{assumption} \label{as:genericposition}
In what follows, we assume that the two projections $p$ and $q$ are in
generic position in the nondegenerate synaptic algebra $A$, i.e.,
\[
p\wedge q=p\wedge q\sp{\perp}=p\sp{\perp}\wedge q=p\sp{\perp}
 \wedge q\sp{\perp}=0, \text{\ and}
\]
\[
p\vee q=p\vee q\sp{\perp}=p\sp{\perp}\vee q=p\sp{\perp}\vee q
 \sp{\perp}=1\not=0.
\]
\end{assumption}

As a consequence of Assumption \ref{as:genericposition}, $p=r\sb{p}$,
$p\sp{\perp}=r\sb{p\sp{\perp}}$, $q=r\sb{q}$, $q\sp{\perp}=
r\sb{q\sp{\perp}}$, and $r=[p,q]=1$, so \emph{we shall have no further
use for $r\sb{p}$, $r\sb{p\sp{\perp}}$, $r\sb{q}$, $r\sb
{q\sp{\perp}}$, $r$, and $[p,q]$}.

Notice that $p$ and $q$ are complements---but not
orthocomplements---in the OML $P$. Likewise for $p$ and $q\sp{\perp}$,
$p\sp{\perp}$ and $q$, and $p\sp{\perp}$ and $q\sp{\perp}$.

\begin{lemma} \label{lm:pqpcarrier}
$(pqp)\dg=(pq\sp{\perp}p)\dg=p,\ $ $(p\sp{\perp}qp\sp{\perp})\dg=(p
\sp{\perp}q\sp{\perp}p\sp{\perp})\dg=p\sp{\perp},\ $ $(qpq)\dg=(qp
\sp{\perp}q)\dg=q,$\ and\  $(q\sp{\perp}pq\sp{\perp})\dg=(q\sp{\perp}p
\sp{\perp}q\sp{\perp})\dg=q\sp{\perp}.$
\end{lemma}

\begin{proof}
By Lemma \ref{lm:carofpqp} (ii), $(pqp)\dg=p\wedge(p\sp{\perp}\vee q)=
p\wedge 1=p$, and the remaining formulas follow similarly.
\end{proof}

\begin{theorem}
The symmetries $u$ and $v$ {\rm(Definition \ref{df:symsuvk})} satisfy
the following conditions{\rm:}
\begin{enumerate}
\item $u$ exchanges $p$ and $q$ as well as $p\sp{\perp}$ and $q\sp{\perp}$.
\item $v$ exchanges $p$ and $q\sp{\perp}$ as well as $q$ and $p\sp{\perp}$.
\end{enumerate}
\end{theorem}

\begin{proof}
In Lemma \ref{lm:uvProps} we have $p\vee q=p\vee q\sp{\perp}=p\sp{\perp}
\vee q=p\sp{\perp}\vee q\sp{\perp}=1$ and it follows that $upu=q$ and
$vpv=q\sp{\perp}$, whence $up\sp{\perp}u=q\sp{\perp}$ and $vqv=p\sp{\perp}$.
\end{proof}

\begin{definition} \label{df:j,ell}
$j:=uvp+pvu$ and $\ell:=2p-1$.
\end{definition}

\begin{theorem} \label{th:j,ellProps}
\
\begin{enumerate}
\item $j$ is a symmetry in $A$ exchanging $p$ and $p\sp{\perp}$.
\item $j$ commutes with both $s$ and $c$.
\item $j=pj+jp$.
\item $\ell=2p-1=p-p\sp{\perp}=cu+sv$ is a symmetry that commutes with $p,c,$
 and $s$.
\end{enumerate}
\end{theorem}

\begin{proof}
(i) Put $x:=uvp\in R$ and $y:=pvu\in R$. Then $j=x+y\in A$. As
$upu=q$, it follows that $up=upu\sp{2}=qu$. Likewise,
$uq\sp{\perp}=p\sp{\perp}u$, and therefore $x=uvp=uq\sp{\perp}v=
p\sp{\perp}uv$. Similarly, $y=pvu=vq\sp{\perp}u=vup\sp{\perp}$.
Consequently, $x\sp{2}=y\sp{2}=0$, $xy=p\sp{\perp}uvvup\sp{\perp}=
p\sp{\perp}$, and $yx=pvuuvp=p$; hence $j\sp{2}=(x+y)\sp{2}=x\sp{2}
+xy+yx+y\sp{2}=p\sp{\perp}+p=1$, so $j$ is a symmetry in $A$. Moreover,
$xp=x$ and $yp=0$, so $jpj=(x+y)p(x+y)=(xp+yp)(x+y)=x(x+y)=x\sp{2}+xy
=p\sp{\perp}$.

(ii) By Definition \ref{df:symsuvk} (2), $s$ commutes with $v$, by
Lemma \ref{lm:uvkCsc}, $s$ commutes with $u$, by Theorem \ref{th:sceffects}
(viii), $s$ commutes with $p$, and it follows that $s$ commutes with
$j=uvp+pvu$. A similar argument shows that $c$ commutes with $j$.

(iii) As $1=p\sp{\perp}+p=jpj+p$, it follows that $j=j\sp{2}pj+jp
=pj+jp$.

(iv) Evidently, $cu+sv=(p-q\sp{\perp})+(p-q)=2p-1=\ell$ and $(2p-1)
\sp{2}=4p-4p+1=1$, so $\ell$ is a symmetry. Obviously, $\ell\in
C(p)\cap C(c)\cap C(s)$.
\end{proof}

\begin{example} \label{ex:matrices3}
{\rm In the rank 2 synaptic algebra in Example \ref{ex:matrices2}, the
projections $p$ and $q$ are in generic position. For the symmetries $u$,
$v$, $j$, and $\ell$, we have}
\[
u=\left[\begin{array}{lr}\cos\,\theta & \sin\,\theta\\ \sin\,\theta &
-\cos\,\theta\end{array}\right],\ v=\left[\begin{array}{lr}\sin\,\theta &
-\cos\,\theta\\ -\cos\,\theta & -\sin\,\theta\end{array}\right],
\]
\[
j=\left[\begin{array}{lr}0 & 1\\ 1 & 0\end{array}\right]\text{\ and\ }
\ell=\left[\begin{array}{lr}1 & 0\\ 0 & -1\end{array}\right].
\]
\end{example}

\begin{lemma} \label{lm:uv+vu}
\
\begin{enumerate}
\item $c\dg=s\dg=(cs)\dg=(csj)\dg=1$.
\item $uv+vu=0$.
\item $j=k$.
\end{enumerate}
\end{lemma}

\begin{proof}
(i) By Theorem \ref{th:carofc&s} (ii), $c\dg=(p\vee q\sp{\perp})
\wedge(p\sp{\perp}\vee q)=1\wedge 1=1$ and $s\dg=(p\vee q)\wedge
(p\sp{\perp}\vee q\sp{\perp})=1\wedge 1=1$. Also, $cs=sc\in A$, so
by \cite[Theorem 5.5]{ComSA}, $(cs)\dg=c\dg s\dg=1$. Moreover, $j\dg
=(j\sp{2})\dg=1\dg=1$, and $csj=jcs$, so by \cite[Theorem 5.5]{ComSA}
again, $(csj)\dg=(cs)\dg j\dg=1$.

(ii) By Theorem \ref{th:j,ellProps} (iv), $1=\ell\sp{2}=(cu+sv)
\sp{2}=(cu)\sp{2}+cs(uv+vu)+(sv)\sp{2}=c\sp{2}+s\sp{2}+cs(uv+vu)=
1+cs(uv+vu)$, whence $cs(uv+vu)=0$, and it follows from (i) that
$0=(cs)\dg(uv+vu)=1(uv+vu)=uv+vu$.

(iii) By Lemma 4.8 and Theorem \ref{th:sceffects} (viii), $c$ and $s$
commute with $u$, $v$, $p$ and each other, whence by Definition \ref
{df:symsuvk}, and direct calculation
\[
csj=cs(uvp+pvu)=csuvp+cspvu=(cu)(sv)p+p(sv)(cu)
\]
\[
=(p+q-1)(p-q)p+p(p-q)(p+q-1)=pq+qp-2pqp=pqp\sp{\perp}+p\sp{\perp}qp.
\]
Also, by Theorem \ref{th:genCSdecomp}, we have $csk=pqp
\sp{\perp}+p\sp{\perp}qp$, and it follows that $cs(k-j)=0$. Thus, by
(i), $k-j=(cs)\dg(k-j)=0$, proving (iii).
\end{proof}

Combining Theorem \ref{th:genCSdecomp}, Theorem \ref{th:j,ellProps},
and Lemma \ref{lm:uv+vu}, we obtain the following synaptic-algebra
version of Halmos' CS-decomposition theorem.

\begin{theorem} [{CS-Decomposition}] \label{th:HalmosCStheorem}
If $p$ and $q$ are projections in generic position in $A$, then
\[
q=c\sp{2}p+csj+s\sp{2}p\sp{\perp},
\]
where $pqp=c\sp{2}p=pc\sp{2}$, $p\sp{\perp}qp\sp{\perp}=s\sp{2}p
\sp{\perp}=p\sp{\perp}s\sp{2}$, $pqp\sp{\perp}+p\sp{\perp}qp=csj$,
$c\dg=s\dg=1$, $j$ is a symmetry exchanging $p$ and $p\sp{\perp}$,
$cCs$, $cCj$, $sCj$, and $j\in CC(pqp\sp{\perp}+p\sp{\perp}qp)$.
\end{theorem}

\noindent In Section \ref{sc:OpMat} we show that Halmos' CS-decomposition
theorem can be derived from Theorem \ref{th:HalmosCStheorem}.

\section{Applications of the CS-decomposition} \label{sc:Apps}

In this section we illustrate the utility of Theorem \ref
{th:HalmosCStheorem} by establishing some results analogous to those
in \cite{Guide} and \cite{Halmos}. Thus, in what follows, \emph{we
assume that $p$ and $q$ are projections in generic position and
that the CS-decomposition of $q$ with respect to $p$ is}
\[
q=c^2p+csj+s^2p^{\perp}.
\]

In the following theorem we use Theorem \ref{th:HalmosCStheorem} to
calculate the spectrum of the sum $p+q$. We denote by $\sigma(a)$ the
spectrum of an element $a\in A$ and, as is customary, we identify
each real number $\lambda\in\reals$ with the element $\lambda 1\in A$.
See \cite[\S 8]{FSynap} for an account of spectral theory in a synaptic
algebra.

\begin{theorem} [{Cf. \cite[Example 2.1]{Guide}}] \label{th:specp+q}
The spectrum of $p+q$ is $\sigma(p+q)=\{1\pm\gamma:\gamma\in\sigma(c)\}$.
\end{theorem}

\begin{proof}
Put $a:=p+q$. Then by Theorem \ref{th:sceffects} (i), $(a-1)^2=
(p+q-1)^2=c^2$, whereupon
\[
\{(\lambda-1)\sp{2}:\lambda\in\sigma(a)\}=\sigma((a-1)\sp{2})
 =\sigma(c\sp{2})=\{\gamma\sp{2}:\gamma\in\sigma(c)\}.
\]
Therefore, for all $\lambda\in\sigma(p+q)=\sigma(a)$, there exists
$\gamma\in\sigma(c)$ such that $\lambda=1+\gamma$ or $\lambda=1-
\gamma$. Moreover, for any $\gamma\in\sigma(c)$ there exists
$\lambda\in\sigma(p+q)$ such that one of the latter two equations
holds.

Let $\gamma\in\sigma(c)$. To complete the proof, it will suffice to
show that $1+\gamma\in\sigma(p+q)$ iff $1-\gamma\in\sigma(p+q)$,
i.e., that $\gamma\in\sigma(p+q-1)$ iff $-\gamma\in\sigma(p+q-1)$.
By the CS-decomposition of $q$ with respect to $p$, we have
\begin{eqnarray*}
p+q-1&=& p+c^2p+s^2p^{\perp}+csj-p-p^{\perp}=c^2p+(s^2-1)p^{\perp}+csj\\
&=& c^2p-c^2p^{\perp}+csj=c^2(p-p^{\perp})+csj=c\sp{2}\ell+csj,
\end{eqnarray*}
where $\ell:=p-p^{\perp}=2p-1$ is a symmetry commuting with $p$, $c$,
and $s$ (Theorem \ref{th:j,ellProps} (iv)); moreover, from $jpj=
p^{\perp}$ we get
\[
j\ell j=-\ell,\ \ell j=-j\ell,\text{\ and\ }\ell j\ell=-j.
\]
Thus, the element $(p+q-1)-\gamma=c^2\ell+csj-\gamma$ is invertible
iff $j(c^2\ell+csj-\gamma)j=-c\sp{2}\ell+csj-\gamma$ is invertible
iff $\ell(-c^2\ell+csj-\gamma)\ell=-c\sp{2}\ell-csj-\gamma=-(p+q-1)
-\gamma$ is invertible. Hence, $\gamma\in\sigma(p+q-1)$ iff
$-\gamma\in\sigma(p+q-1)$.
\end{proof}

We now turn our attention to some commutativity results that
involve the CS-decomposition.

\begin{lemma}\label{lm:commutant}
Suppose that there exists $b=bp=pb
\in C(c)$ and that $a=b+jbj$. Then $ap=pa=b$ and $a\in C(q)$.
\end{lemma}

\begin{proof} Assume the hypotheses of the lemma. Then
\setcounter{equation}{0}
\begin{equation} \label{eq:CommuteLemma1}
bp\sp{\perp}=p\sp{\perp}b=0, \text{\ so\ } jbjp=jbp\sp{\perp}j
 =0=jp\sp{\perp}bj=pjbj,
\end{equation}
whence
\begin{equation} \label{eq:CommuteLemma2}
jbjp\sp{\perp}=p\sp{\perp}jbj=jbj,\text{\ and\ }ap=bp+jbjp=b
=pb+pjbj=pa.
\end{equation}
Since $b\in C(c)$, we have $b\in C(s)$ by Theorem \ref
{th:sceffects} (viii). Using the data in (\ref{eq:CommuteLemma1})
and (\ref{eq:CommuteLemma2}), we find that
\[
aq=b(pc\sp{2}+csj+p\sp{\perp}s\sp{2})+jbj(pc\sp{2}+jcs+p\sp{\perp}
 s\sp{2})=c^2b+csbj+jbcs+jbjs^2,
\]
whereas
\[
qa=(c\sp{2}p+jcs+s\sp{2}p\sp{\perp})b+(c\sp{2}p+csj+s\sp{2}p\sp{\perp})
 jbj=c\sp{2}b+jbcs+csbj+s\sp{2}jbj.
\]
Since $s\sp{2}Cjbj$, it follows that $aq=qa$.
\end{proof}

\begin{theorem}\label{th:commutant}
Let $z\in P$ be a projection. Then $z\in C(p)\cap C(q)$ iff
there exists a projection $t\in P$ such that $t=tp=pt\in C(c)$ and
$z=t+jtj$.
\end{theorem}

\begin{proof}
If $t\in P$, $t=tp=pt\in C(c)$, and $z=t+jtj$, then $z\in C(p)
\cap C(q)$ by Lemma \ref{lm:commutant} with $a:=z$ and $b:=t$.

Conversely, suppose that $z\in P\cap C(p)\cap C(q)$ and let $g:=
|p-z\sp{\perp}|$ be the cosine effect of the projection $z$ with
respect to $p$ (Theorem \ref{th:sceffects} (iv)). Thus, $g\in C(p)$
and since $z$ commutes with $p$, we infer from Corollary \ref{co:pCq}
(vi) that $g$ is a projection and $z=gp+g\sp{\perp}p\sp{\perp}$. By
Theorem \ref{th:sceffects} (ix), $z\in C(c)$. Moreover, as $p,z\in
C(c)$, we have $p-z\sp{\perp}\in C(c)$, whence $g=|p-z^{\perp}|
\in C(c)$. Also, since $j\in CC(pqp\sp{\perp}+p\sp{\perp}qp)$ and
$z\in C(p)\cap C(q)$, it follows that $j\in C(z)$. From this and from
$p\sp{\perp}j=jp$ we find that
\[
gp+g\sp{\perp}p\sp{\perp}=z=jzj\\=jgpj+jg\sp{\perp}p\sp{\perp}j\\=
 jgjp\sp{\perp}+jg\sp{\perp}jp,
\]
and multiplying both sides of the last equation by $p\sp{\perp}$
from the right, we obtain $g\sp{\perp}p\sp{\perp}=jgjp\sp{\perp}$.
Consequently,
\[
z=gp+g\sp{\perp}p\sp{\perp}=gp+jgjp\sp{\perp}=gp+jgpj.
\]
Now  put $t:=gp=pg$.  Then $z=t+jtj$, $tp=pt=t$, and $t\sp{2}=
g\sp{2}p=gp=t$, so $t\in P$. Moreover, since $g,p\in C(c)$, it
follows that $t\in C(c)$.
\end{proof}

In the next theorem, we find  conditions under which an
arbitrary element $a\in A$ commutes with projections $p$ and
$q$ in generic position.  The limits in the proof are
taken with respect to the order-unit norm on $A$ \cite
[p. 634]{FSynap}.

\begin{theorem}\label{commutantgen} An element $a\in A$
commutes with both $p$ and $q$ iff there exists $b\in C(c)$ such that
$b=bp=pb$ and $a=b+jbj$.
\end{theorem}

\begin{proof} If $b\in C(c)$, $b=bp=pb$, and $a=b+jbj$, then
$a\in C(p)\cap C(q)$ by Lemma \ref{lm:commutant}.

Conversely, assume that $a\in C(p)\cap C(q)$ and let $(z\sb{\lambda})
\sb{\lambda\in\reals}$ be the spectral resolution of $a$ (\cite
[Definition 8.2 (ii)]{FSynap}). By \cite[Theorem 8.10]{FSynap},
$z\sb{\lambda}\in P\cap C(p)\cap C(q)$ for all $\lambda\in\reals$,
whence by Theorem \ref{th:commutant}, for each $\lambda\in\reals$,
there exists a projection $t\sb{\lambda}\in P$ such that $t\sb{\lambda}
=t\sb{\lambda}p=pt\sb{\lambda}\in C(c)$ and $z\sb{\lambda}=
t\sb{\lambda}+jt\sb{\lambda}j$.

By \cite[Corollary 8.6]{FSynap}, there is an ascending sequence
$a\sb{1}\leq a\sb{2}\leq\cdots$ in $CC(a)$ such that $a=\lim\sb
{n\rightarrow\infty}a\sb{n}$ and each $a\sb{n}$ is a finite real
linear combination of projections $z\sb{\lambda}$ in the spectral
resolution of $a$. Let $n$ be a positive integer. Then, since
$a\sb{n}\in CC(a)$ and $a\in C(p)\cap C(q)$, it follows that
$a\sb{n}\in C(p)\cap C(q)$. Moreover,
\[
a\sb{n}=\sum\sb{i=1}\sp{M\sb{n}}\alpha\sb{n,i}z\sb{\lambda\sb{n,i}}
 =\sum\sb{i=1}\sp{M\sb{n}}\alpha\sb{n,i}(t\sb{\lambda\sb{n,i}}+
 jt\sb{\lambda\sb{n,i}}j)=d\sb{n}+jd\sb{n}j,
\]
where $\alpha\sb{n,i}\in\reals$ and $d\sb{n}:=\sum\sb{i=1}\sp{M\sb{n}}
\alpha\sb{n,i}t\sb{\lambda\sb{n,i}}$. Since $t\sb{\lambda\sb{n,i}}=
t\sb{\lambda\sb{n,i}}p=pt\sb{\lambda\sb{n,i}}\in C(c)$, we have
\[
d\sb{n}=d\sb{n}p=pd\sb{n}\in C(c),\text{\ and\ }jd\sb{n}j=jd\sb{n}pj=
 jd\sb{n}jp\sp{\perp}.
\]
Thus, $a\sb{n}p=pa\sb{n}=d\sb{n}p+jd\sb{n}jp=d\sb{n}\in C(c)$. Put $b:
=ap=pa$, noting that $b=bp=pb$. Also, since $a,p\in C(c)$, we have
$b\in C(c)$. Moreover,
\[
b=ap=(\lim\sb{n\rightarrow\infty}a\sb{n})p=\lim\sb{n\rightarrow\infty}
 (a\sb{n}p)=\lim\sb{n\rightarrow\infty}d\sb{n}.
\]
By \cite[Theorem 8.11]{FSynap}, $C(c)$ is closed in the order-unit-norm
topology, whence $b=\lim\sb{n\rightarrow\infty}d\sb{n}\in C(c)$. Moreover,
since $j$ is a symmetry, we have
\[
jbj=j(\lim\sb{n\rightarrow\infty}d\sb{n})j=\lim\sb{n\rightarrow\infty}
 (jd\sb{n}j),
\]
and it follows that
\[
a=\lim\sb{n\rightarrow\infty}a\sb{n}=\lim\sb{n\rightarrow\infty}
 (d\sb{n}+jd\sb{n}j)=b+jbj. \qedhere
\]
\end{proof}

\section{Operator-matrix consequences} \label{sc:OpMat}

Let ${\mathcal H}$ be a nonzero complex separable Hilbert space,
let ${\mathcal B}({\mathcal H})$ be the C$\sp{\ast}$-algebra of
all bounded linear operators on ${\mathcal H}$, and let $A$ be the
synaptic algebra of all self-adjoint operators in ${\mathcal B}
({\mathcal H})$. The assumption that $p,q\in A$ are projections in
generic position is still in force. Following the notation in
\cite[pp. 1413 and ff.]{Guide}, we define the following closed
linear subspaces of ${\mathcal H}$:
\[
M\sb{0}:=p({\mathcal H})\text{\ and\ }M\sb{1}:=p\sp{\perp}
 ({\mathcal H}).
\]
Then ${\mathcal H}=M\sb{0}\oplus M\sb{1}$; hence, in what follows,
\emph{we shall regard each vector $h\in{\mathcal H}$ as having the
form}
\[
h=\left[\begin{array}{c}x\\w\end{array}\right],\text{\ where\ }x
 \in M\sb{0},\ w\in M\sb{1},\ ph=\left[\begin{array}{c}x\\0
 \end{array}\right]\text{\ and\ }\ p\sp{\perp}h=\left[\begin{array}
 {c}0\\w\end{array}\right].
\]
By Theorem \ref{th:j,ellProps}, $p$ and $p\sp{\perp}$ are exchanged by a
symmetry $j$ in $A$. Thus, for each $w\in M\sb{1}$,
\[
j\left[\begin{array}{c}0\\w\end{array}\right]=jp\sp{\perp}\left
[\begin{array}{c}0\\w\end{array}\right]=pj\left[\begin{array}{c}0\\w
\end{array}\right],
\]
whence there is a uniquely determined element $Rw\in M\sb{0}$ such
that
\[
j\left[\begin{array}{c}0\\w\end{array}\right]=\left[\begin{array}
{c}Rw\\0\end{array}\right].
\]
It is not difficult to show that $R\sp{\ast}=R\sp{-1}$, i.e.,
$R\colon M\sb{1}\to M\sb{0}$ is a unitary isomorphism, and
$R\sp{\ast}\colon M\sb{0}\to M\sb{1}$ satisfies
\[
j\left[\begin{array}{c}x\\0\end{array}\right]=\left[\begin{array}
 {c}0\\R\sp{\ast}x\end{array}\right],\text{\ whence\ }j\left
 [\begin{array}{c}x\\w\end{array}\right]=\left[\begin{array}
 {c}Rw\\R\sp{*}x\end{array}\right]=\left [\begin{array}{cc} 0 & R
 \\ R\sp{\ast} & 0\end{array}\right]\left[\begin{array}{c}x\\w
 \end{array}\right].
\]
Thus in operator-matrix form,
\[
j=\left [\begin{array}{cc} 0 & R\\ R\sp{\ast} & 0\end{array}
 \right]\text{\ and\ \ }p=\left [\begin{array}{cc} I & 0\\ 0 & 0
 \end{array}\right],
\]
where $I\colon M\sb{0}\to M\sb{0}$ is the identity operator.

Let
\[
{\mathcal K}:=\left\{ \left[\begin{array}{c}x\\y\end{array}\right]
 :x,y\in M\sb{0}\right\}
\]
be organized into a Hilbert space in the obvious way.  It is easy
to verify that the mappings given in operator-matrix form by
\[
\left[\begin{array}{cc} I & 0\\ 0 & R\end{array}\right]
 \left[\begin{array}{c}x\\w\end{array}\right]=\left[\begin{array}{c}x\\
 Rw\end{array}\right]\text{\ and\ }\left[\begin{array}{cc} I & 0\\ 0 &
 R\sp{\ast}\end{array}\right]\left[\begin{array}{c}x\\y\end{array}\right]
 =\left[\begin{array}{c}x\\R\sp{\ast}y\end{array}\right]
\]
for
\[
\left[\begin{array}{c}x\\w\end{array}\right]\in{\mathcal H}\text{\ and\ }
\left[\begin{array}{c}x\\y\end{array}\right]\in{\mathcal K}
\]
are inverse unitary isomorphisms of ${\mathcal H}$ onto ${\mathcal K}$
and of ${\mathcal K}$ onto ${\mathcal H}$.

Since the projections $p$ and $q$ are in generic position, Theorem
\ref{th:HalmosCStheorem} yields the decomposition
\[
q=c\sp{2}p+csj+s\sp{2}p\sp{\perp},
\]
where $c$ and $s$ are the sine and cosine effects of $q$ with respect to
$p$, $pqp=c\sp{2}p=pc\sp{2}$, $p\sp{\perp}qp\sp{\perp}=s\sp{2}p
\sp{\perp}=p\sp{\perp}s\sp{2}$, $cCp$, $sCp$, $cCs$, and as above,
$j$ is a symmetry exchanging $p$ and $p\sp{\perp}$. Moreover, $pqp
\sp{\perp}+p\sp{\perp}qp=csj$, $cCj$, $sCj$, and $j\in CC(pqp\sp
{\perp}+p\sp{\perp}qp)$.

For each $x\in M\sb{0}$,
\[
c\left[\begin{array}{c}x\\0\end{array}\right]=cp\left[\begin{array}
 {c}x\\0\end{array}\right]=pc\left[\begin{array}{c}x\\0\end{array}
 \right]\text{\ and\ }s\left[\begin{array}{c}x\\0\end{array}\right]
 =sp\left[\begin{array}{c}x\\0\end{array}\right]=ps\left[\begin{array}
 {c}x\\0\end{array}
 \right]
\]
whence there are uniquely determined elements $Cx\in M\sb{0}$ and
$Sx\in M\sb{0}$ such that
\[
c\left[\begin{array}{c}x\\0\end{array}\right]=\left[\begin{array}
{c}Cx\\0\end{array}\right]\text{\ and\ }s\left[\begin{array}{c}x\\0
 \end{array}\right]=\left[\begin{array}{c}Sx\\0\end{array}\right].
\]
Using the fact that $cCj$, we have, for all $w\in M\sb{1}$,
\[
c\left[\begin{array}{c}0\\w\end{array}\right]=cj\left[\begin{array}{c}
 Rw\\0\end{array}\right]=jc\left[\begin{array}{c}Rw\\0\end{array}\right]
 =j\left[\begin{array}{c}CRw\\0\end{array}\right]=\left[\begin{array}
 {c}0\\R\sp{\ast}CRw\end{array}\right],
\]
whence for $x\in M\sb{0}$ and $w\in M\sb{1}$,
\[
c\left[\begin{array}{c}x\\w\end{array}\right]=\left[\begin{array}{c}
 Cx\\R\sp{\ast}CRw\end{array}\right],\text{\ similarly }s\left
 [\begin{array}{c}x\\w\end{array}\right]=\left[\begin{array}{c}
 Sx\\R\sp{\ast}SRw\end{array}\right],
\]
and it follows that
\[
csj\left[\begin{array}{c}x\\w\end{array}\right]=cs\left[\begin{array}
{c}Rw\\R\sp{\ast}x\end{array}\right]=\left[\begin{array}{c}CSRw\\
R\sp{\ast}CSRR\sp{\ast}x\end{array}\right]=\left[\begin{array}{c}
CSRw\\R\sp{\ast}CSx\end{array}\right].
\]

Using the properties of $c$ and $s$, it is not difficult to show that
$C$ and $S$ are self-adjoint operators on $M\sb{0}$, $0\leq C\leq I$,
$0\leq S\leq I$, $C\sp{2}+S\sp{2}=I$, and that $C$ and $S$ have kernel
zero. We note that, in operator-matrix form,
\[
c=\left[\begin{array}{cc} C & 0\\ 0 & R^*CR\end{array}\right]=\left[
 \begin{array}{cc}I & 0\\ 0 & R\sp{\ast}\end{array}\right]\left[\begin{array}
 {cc} C & 0\\ 0 & C\end{array}\right]\left[\begin{array}{cc} I & 0\\
 0 & R\end{array}\right]\text{\ and}
\]
\[
s=\left[\begin{array}{cc} S & 0\\ 0 & R^*SR\end{array}\right]=\left[
 \begin{array}{cc}I & 0\\ 0 & R\sp{\ast}\end{array}\right]\left[\begin
 {array}{cc}S & 0\\0 & S\end{array}\right]\left[\begin{array}{cc}I & 0\\
  0 & R\end{array}\right].
\]
Similarly, it is not difficult to show that the symmetries $u$ and $v$
can be expressed in matrix-operator form as
\[
u=\left[\begin{array}{cc} C & SR\\ R^*S & -R^*CR\end{array}\right]=\left[
 \begin{array}{cc}I & 0\\ 0 & R\sp{\ast}\end{array}\right]\left[\begin{array}
 {cc} C & S\\ S & -C\end{array}\right]\left[\begin{array}{cc} I & 0\\
 0 & R\end{array}\right]\text{\ and}
\]
\[
v=\left[\begin{array}{cc} S & -CR\\ -R^*C & -R^*SR\end{array}\right]=\left[
 \begin{array}{cc}I & 0\\ 0 & R\sp{\ast}\end{array}\right]\left[\begin{array}
 {cc}S & -C\\-C & -S\end{array}\right]\left[\begin{array}{cc} I & 0\\
 0 & R\end{array}\right].
\]

In view of the results above,
\[
q\left[\begin{array}{c}x\\w\end{array}\right]=(c\sp{2}p+csj+s\sp{2}
 p\sp{\perp})\left[\begin{array}{c}x\\w\end{array}\right]=\left
 [\begin{array}{c}C\sp{2}x+CSRw\\R\sp{\ast}CSx+R\sp{\ast}SRw\end{array}
 \right]
\]
\[
=\left[\begin{array}{cc} I & 0\\ 0 & R\sp{\ast}\end{array}\right]\left
 [\begin{array}{cc} C\sp{2} & CS\\ CS & S\sp{2}\end{array}\right]\left
 [\begin{array}{cc} I & 0\\ 0 & R\end{array}\right]\left[\begin{array}
 {c}x\\w\end{array}\right],
\]
whereupon, in operator-matrix form,
\[
q=\left[\begin{array}{cc} I & 0\\ 0 & R\sp{\ast}\end{array}\right]\left
 [\begin{array}{cc} C\sp{2} & CS\\ CS & S\sp{2}\end{array}\right]\left
 [\begin{array}{cc} I & 0\\ 0 & R\end{array}\right].
\]
This is precisely \cite[Theorem 1.1]{Guide}; hence, \emph{Theorem
\ref{th:HalmosCStheorem} is a true generalization of Halmos'
CS-decomposition theorem.}

Now let $a\in A$. Then by Theorem \ref{commutantgen}, $a\in C(p)
\cap C(q)$ iff there exists $b\in C(c)$ such that $a=b+jbj$ and
$b=bp=pb$. Evidently, $b=bp=pb$ iff $b$ has the operator-matrix form
\[
b=\left[\begin{array}{cc} B & 0\\ 0 & 0\end{array}\right],
\]
where $B$ is a self-adjoint operator on $M\sb{0}$, in which case
\[
jbj=\left[\begin{array}{cc}0 & R\\R\sp{\ast} & 0\end{array}\right]
 \left[\begin{array}{cc}B & 0\\ 0 & 0\end{array}\right]\left[\begin
 {array}{cc}0 & R\\ R\sp{\ast}& 0\end{array}\right]=\left[\begin
 {array}{cc}0 & 0\\ 0 & R\sp{\ast}BR\end{array}\right]
\]
and $bCc$ iff $BC=CB$. Moreover, in operator-matrix form, the
condition $a=b+jbj$ is
\[
a=\left[\begin{array}{cc} B & 0\\ 0 & 0\end{array}\right]+\left
 [\begin{array}{cc} 0 & 0\\ 0 & R\sp{\ast}BR\end{array}\right]
 =\left[\begin{array}{cc} 0 & 0\\ 0 & R\sp{\ast}\end{array}\right]
 \left[\begin{array}{cc} B & 0\\ 0 & B\end{array}\right]\left[
 \begin{array}{cc} 0 & 0\\ 0 & R\end{array}\right].
\]
This is precisely Halmos' solution of the problem of finding
the simultaneous commutant of two projections in generic
position \cite[p.385]{Halmos}.

\end{document}